\documentclass[11pt, english]{amsart}
\usepackage{amssymb,amsthm,amsmath,amstext,graphicx}
\usepackage{mathrsfs}
\usepackage{bm}
\usepackage[utf8]{inputenc}
\usepackage{amscd}
\usepackage{comment}
\usepackage{youngtab}
\usepackage[all]{xy}

\newcommand{\jmq}[1]{ {\color{cyan} #1} }

\makeatletter
\ifnum\@ptsize=0  \fi \ifnum\@ptsize=2  \fi 

\title{Matrices of linear forms of constant rank from vector bundles 
on projective spaces}
\author[L. Manivel]{Laurent Manivel}
\address{Paul Sabatier University
118, route de Narbonne.
F-31062 Toulouse Cedex~9.
France}
\email{manivel@math.cnrs.fr}
 \author[R.\ M.\ Mir\'o-Roig]{Rosa M.\ Mir\'o-Roig} 
  \address{Facultat de
  Matem\`atiques i Inform\`atica, Universitat de Barcelona, Gran Via des les
  Corts Catalanes 585, 08007 Barcelona, Spain} \email{miro@ub.edu, ORCID 0000-0003-1375-6547}

\date{\today}

\theoremstyle{plain}
\newtheorem{theorem}{Theorem}

\newtheorem{prop}[theorem]{Proposition}
\newtheorem{lemma}[theorem]{Lemma}
\newtheorem{coro}[theorem]{Corollary}

\def\CC{{\mathbb{C}}}

\def\PP{{\mathbb{P}}}
\def\ZZ{{\mathbb{Z}}}

\def\cO{{\mathcal{O}}}
\def\cE{{\mathcal{E}}}
\def\cF{{\mathcal{F}}}\def\cX{{\mathcal{X}}}
\def\cG{{\mathcal{G}}}\def\cK{{\mathcal{K}}}
\def\cL{{\mathcal{L}}}
\def\cR{{\mathcal{R}}}

\def\cC{{\mathcal{C}}}\def\cQ{{\mathcal{Q}}}
\def\cM{{\mathcal{M}}}
\def\ra{{\rightarrow}}\def\lar{{\leftarrow}}
\def\lra{{\longrightarrow}}

\def\fsl{\mathfrak{sl}}

\def\om{\omega}

\usepackage{xcolor}
\def\lau#1{\textcolor{green}{\;{\bf Laurent:} #1 {\bf }}}

 \def\cD{{\mathcal D}}

\def\BC{\mathbb C}

\def\BP{\mathbb P}
\def\pp#1{\mathbb P^{#1}}

\def\pp#1{{\mathbb P}^{#1}}
\def\tdim{{\rm dim}}

\def\ww{\wedge}

\def\cA{{\mathcal A}}

\def\cE{{\mathcal E}}
\def\cF{{\mathcal F}}
\def\cG{{\mathcal G}}
\def\cR{{\mathcal R}}

\def\cL{{\mathcal L}}

\def\cO{{\mathcal O}}
\def\CC{\mathbb C}

\def\ZZ{\mathbb Z}

\def\11{\mathbf 1}
\def\PP{\mathbb P}

\def\fsl{{\mathfrak {sl}}}

\def\a{\alpha}

\def\ot{{\mathord{ \otimes } }}
\def\op{{\mathord{\,\oplus }\,}}

\def\lra{{\mathord{\;\longrightarrow\;}}}
\def\ra{{\mathord{\;\rightarrow\;}}}

\def\dim{{\rm dim}\;}
\def\La#1{\Lambda^{#1}}

\def\frak{\mathfrak}
\def\fsl{\frak s\frak l}

\def\op{\oplus}

\def\op{\oplus}

\def\ul{\underline}

\def\t{\tau}

\def\a{\alpha}

\def\BP{\mathbb  P}
\def\BC{\mathbb  C}

\def\pp#1{\mathbb  P^{#1}}
\def\cC{\mathcal  C}

\def\cQ{\mathcal  Q}

\def\hd{, \dotsc ,}

\def\La#1{\Lambda^{#1}}

\def\pp#1{\mathbb  P^{#1}}

\def\ur{\underline{\mathbf{R}}}

\def\ra{\rightarrow}

\def\tdim{\operatorname{dim}}

\def\tmod{\operatorname{mod}}

\def\thom{\operatorname{Hom}}

\def\ww{\wedge}

\def\be{\begin{equation}}
\def\ene{\end{equation}}

\newcommand{\isom}{\cong}

\def\cK{{\mathcal K}}

\newcommand{\red}[1]{{\color{red}#1}}

\theoremstyle{remark}
\newtheorem{remark}[theorem]{Remark}
\newtheorem{example}[theorem]{Example}

\begin{document}

\begin{abstract}
    We consider the problem of constructing matrices of linear forms of constant rank by focusing on the associated vector bundles on projective spaces. Important examples are given by the classical Steiner bundles, as well as 
    some special (duals of) syzygy bundles that we call Drézet bundles. Using the 
    classification of globally generated vector bundles with small first Chern class
    on projective spaces, we are able to describe completely the indecomposable 
    matrices of constant rank up to six; some of them come from rigid homogeneous vector bundles, some other from Drézet bundles related either to plane quartics or to instanton bundles on $\PP^3$. 
\end{abstract}

\maketitle

It is a classical problem in both algebraic geometry and linear algebra  to construct linear spaces of matrices of 
constant rank (constant outside the origin).  
The seminal modern reference for this problem is \cite{MR954659}, as well as for the closely related problem of constructing linear spaces of matrices of bounded
rank (in the sense that the rank is never maximal). It was followed 
by an important literature, to which this paper is a contribution. 

Two dimensional spaces of matrices (including those of constant or bounded rank,
of course) were classified by Kronecker and Weierstrass. 
Atkinson \cite{MR695915} classified spaces of bounded rank at most three.
He showed that there is a unique primitive  three dimensional space of 
constant rank two, $\La 2\BC^3\subset End(\BC^3)$, 
and a unique primitive four dimensional space of constant rank three 
corresponding to the natural  inclusion
$\BC^4\hookrightarrow \thom(\BC^{4}, \La 2 \BC^4)$.
All other spaces are obtained by restricting these two subspaces.
More generally, $\CC^n\hookrightarrow \thom(\La k \CC^n,\La {k+1}\CC^n)$ has always constant rank. This example, dating at least back to Westwick \cite{MR878293}, was vastly generalized in \cite{landsberg-manivel-bounded}. 

Beyond that, only  isolated examples were known.
Westwick in \cite{MR1374258} produced a $4$-dimensional space of 
skew-symmetric $10\times 10$ matrices of constant rank $8$ that remained mysterious
until 2012 when Boralevi, Faenzi, and Mezzetti \cite{MR3107531} gave an explanation for it in terms of instanton bundles and generalized it to a family of such. 
They also  found   four dimensional families
of $14\times 14$ skew-symmetric matrices of constant rank $12$.

Eisenbud and Harris \cite{MR954659} 
also observed that spaces of syzygies furnish spaces of bounded rank, at least in 
the case where the syzygies in some degree are only linear. We can even get 
matrices of constant rank by restricting to suitable linear subspaces. 
Many instances of such linear syzygies are known, suggesting that the classification of matrices of constant rank must be  a wild problem. 

\medskip
There is a strong 
relationship with the study of vector bundles on projective spaces, another classical topic 
that attracts considerable attention. Indeed, a vector space of dimension $n+1$ of matrices
of size $a\times b$ can be seen as a matrix with linear entries, 
or equivalently as a 
morphism of sheaves $\psi: \cO_{\PP^n}^{\oplus a}\lra \cO_{\PP^n}(1)^{\oplus b}$. If the rank is constant, equal to $e$, the image of this 
morphism is a vector bundle $\cE$ of rank $e$. Letting $\cK$ and $\cC$ denote the kernel and cokernel bundles, respectively,
we get a diagram

$$\xymatrix{
& \cK \ar[rd] &&& & \cC & \\
& & \cO_{\PP^n}^{\oplus a}\ar[rr]^\psi\ar[rd] & 
&\cO_{\PP^n}(1)^{\oplus b}\ar[ru] & &\\
& & & \cE\ar[ru]& && \\
 }$$
 
\noindent whose diagonals are short exact sequences. The vector bundle $\cE$ has very special properties, in particular:
\begin{itemize}
    \item $\cE$ and $\cE^\vee(1)$ are generated by global sections;
    \item as a consequence, $\cE$ is uniform, in the sense that its restriction to every line $L\subset\PP^n$ 
    splits in the same way, as
    $$\cE_{|L}\simeq \cO_L(1)^{\oplus c_1(E)}\oplus \cO_L^{\oplus (e-c_1(E))}.$$
\end{itemize}
Uniform vector bundles have been classified up to rank $e\le n+1$. In this range, for $n\ge 3$ they are sums of 
line bundles and the tautological quotient bundle $Q$ or its dual (see \cite{MR3536970} and references therein. In that paper the same result is conjectured to hold for $n\ge 5$ and $e<2n$).   The general philosophy is that there should exist very few 
uniform vector bundles on $\PP^n$ of small rank, but they are much easier to construct when the 
rank is large. 
Conversely, if $\cE$ is a rank $e$ vector bundle on $\PP^n$, such that 
$\cE$ and $ \cE^\vee(1)$ are generated by global sections (so that in particular $\cE$ is uniform), 
the natural morphism
$$\psi_\cE: H^0(\PP^n,\cE)\otimes\cO_{\PP^n}\longrightarrow H^0(\PP^n,\cE^\vee(1))^\vee\otimes\cO_{\PP^n}(1)$$
has constant rank $e$. This yields a nice recipe to construct matrices 
of constant rank, that we have applied concretely on a series of significant examples.

\medskip 
The point of view chosen in this paper is to understand families of matrices of constant rank in terms of  the corresponding vector bundles and their moduli spaces. Indeed, if we have family of matrices of constant rank, parametrized by a certain scheme $\cX$, there is an induced morphism $\cX\lra {\mathcal Quot}(\PP^n)$, mapping $x\in \cX$ to a rank $e$ vector
bundle $\cE_x$ on $\PP^n$. If moreover $ \cE_x$ is semistable for each $x\in \cX$, we even get a 
morphism $\cX\lra\cM(e,c)$ to the coarse moduli space of rank $e$ semistable vector bundles on $\PP^n$ with 
 total Chern class $c$. Conversely, we deduce that there is an open subset of the space of matrices of 
constant rank $e$ that is essentially parametrized by some locally closed subset 
$\cM(e,c)_{gen}$ of  $\cM(e,c)$. 
Indeed, after restricting if necessary
to an open subset of the family, we may suppose that $h=h^0(\cE_x)$ and $k=h^0(\cE_x^\vee(1))$ are constant. Then locally, 
$\psi_{\cE_x}$ only deforms to morphisms of the same type among spaces  of matrices of size
$h\times k$ and of  constant rank $e$. A similar conclusion holds for matrices of size $a\times b$ with $a\le h$ and $b\le k$
giving rise to vector bundles from the same family; we get a moduli space 
which maps to $\cM(e,c)_{gen}$
with fibers that are open subsets of the product of Grassmannians $Gr(a,h)\times Gr(b,k)$.\medskip 

Our main results are classification theorems for matrices of linear forms of
constant rank $e\le 6$. To obtain these results we 
build on the classifications of globally generated vector bundles on projective spaces with small first Chern class that were 
obtained some years ago \cite{anghel-manolache, manolache, anghel-coanda-manolache, su1, su2}. Given such a vector bundle $\cE$, the main problem becomes: can $\cE^\vee(1)$ 
also be globally generated? Important instances are given by the so-called 
Steiner bundles, defined by short exact sequences of type 
$$0\ra\cO_{\PP^n}(-1)^{\oplus c}\ra \cO_{\PP^n}^{\oplus e+c}\ra \cE\ra 0,$$
a family of vector bundles that has already been well-studied \cite{MR1240599, miroroig-marchesi}. A closely related family is that of vector bundles defined by an
exact sequence of type 
$$0\ra\cO_{\PP^n}(-c)\ra \cO_{\PP^n}^{\oplus e+1}\ra \cE\ra 0,$$
which are special classes of (duals of) syzygy bundles; we call them
Drézet bundles because they were used by Drézet to construct examples 
of uniform, non-homogeneous vector bundles \cite{drezet-uniform}. After a section 
of preliminaries, we study Steiner and Drézet bundles under our perspective. 
Then we focus on vector bundles with first Chern class $c=2$ and we deduce a classification 
of matrices of linear forms of constant rank four or five.

\begin{theorem}
Given an indecomposable matrix of linear forms of constant rank $4$, 
the associated vector bundle $\cE$ must be, up to switching with $\cE^\vee(1)$, 
either 
\begin{enumerate}
    \item ($c_1=1$) the quotient bundle $Q$ on $\PP^4$, or
    \item ($c_1=2$)  the unique (up to isomorphism) Drézet bundle of rank
    four on $\PP^2$ defined by a non-degenerate conic. 
\end{enumerate}
For an indecomposable matrix of linear forms of constant rank $5$, 
the associated vector  bundle $\cE$ must be, up to switching with $\cE^\vee(1)$, 
\begin{enumerate}
    \item ($c_1=1$) the quotient bundle $Q$ on $\PP^5$, or
    \item ($c_1=2$)  the unique homogeneous  rank
    five Drézet bundle on $\PP^2$. 
\end{enumerate}
\end{theorem}

The corresponding matrices are described in subsection 4.2 and 4.3. Note the surprising fact that 
in rank five, the possible vector  bundles are unique. In rank four they are unique 
up to isomorphism. Also the homogeneous bundles considered in \cite{landsberg-manivel-bounded} play an important r\^ole in the classification. But when the rank increases there are certainly many more possibilities. 
Already in rank six some vector bundles with non-trivial moduli spaces 
appear in the classification:

\begin{theorem}
Given an indecomposable matrix of linear forms of constant rank $6$, 
the associated vector
bundle $\cE$ must be, up to switching with $\cE^\vee(1)$, of one the following types:
\begin{enumerate}
    \item ($c_1=1$) the quotient bundle $Q$ on $\PP^6$, or
    \item ($c_1=2$) a Drézet bundle on $\PP^3$ defined by a net of quadrics, or
    \item ($c_1=3$) the second exterior power $\wedge^2Q$ 
    on $\PP^4$, 
    \item ($c_1=3$)  a Drézet bundle on $\PP^2$ defined by the cubics 
    apolar to a plane quartic, or
    \item ($c_1=3$)  a non-trivial extension of the quotient bundle $Q$ on $\PP^2$ 
    by a rank 4 non-generic Drézet bundle.
\end{enumerate}
\end{theorem}

Cases (1) and (3) are homogeneous and rigid, while cases (2) and (5) have interesting moduli. 
Indeed they are both closely related to plane quartics, and (4) has also
nice unexpected connections with instanton bundles of charge $3$ on $\PP^2$.  
Case (5) is the most surprising one, being built from a rank 4
Drézet bundle that gives rise to a matrix of linear forms of bounded but
not constant rank: a defect that the non-triviality of the extension 
turns out to correct. 

For all these cases we provide explicit matrices of linear forms of constant 
rank. In particular this will illustrate the nice fact that although we focus
on the abstract study of the associated vector bundles, the method is actually quite concrete.

\medskip\noindent {\it Acknowledgements}.  We thank  J.M. Landsberg for initiating this project, as well as for useful comments. We also thank D. Faenzi for his comments and hints. The second author has been partially supported by the grant PID2019-104844GB-I00, and the first one  by the ANR 
project FanoHK, grant ANR-20-CE40-0023.

\section{Preliminaries} 

\subsection{Generalities}
Giving a matrix of linear forms in $n+1$ variables, of size $a\times b$ and constant rank $e$, 
up to linear combinations of rows and columns, is equivalent to providing the following data:
\begin{itemize}
\item a vector bundle $\cE$ of rank $e$ on $\PP^n$, such that $\cE$ and $\cE^\vee(1)$
are both generated by global sections;
\item a subspace $A$ of $H^0(\PP^n,\cE)$, of dimension $a$, generating $\cE$ at every point;
\item a subspace $B$ of $H^0(\PP^n,\cE^\vee(1))$, of dimension $b$, generating $\cE^\vee(1)$ 
at every point.
\end{itemize}
The most important datum is obviously the vector bundle $\cE$. Once it is given, one can 
wonder how many sections are necessary to generate it at every point. By general principles, 
$e+n$ general sections will always suffice. A more precise answer is provided by Segre classes. Indeed, given a subspace $A$ of $H^0(\PP^n,\cE)$ of dimension $a$, the evaluation morphism for sections of $\cE$ restricts to a morphism
$$ev_A: A\otimes \cO_{\PP^n}\longrightarrow \cE.$$
Suppose $a\ge e$. If $ev_A$ is surjective, the kernel $\cK$ is a vector bundle of rank $a-e$, hence $c_{a-e+1}(\cK)=0$ and therefore $s_{a-e+1}(\cE)=0$. In general, the locus $D_A$ where $ev_A$ fails to be surjective has expected codimension 
$a-e+1$, and this expected codimension is the correct one for $A$ generic. In particular, 
$ev_A$ is everywhere surjective if $A$ is generic of dimension $a\ge e+n$. When $a<e+n$
and $D_A$ has the expected codimension, the Thom-Porteous formula asserts that the degree 
of $D_A$ is the degree of the Segre class $s_{a-e+1}(\cE)$. This applies in particular when $D_A$ 
is empty, and we get a contradiction if $s_{a-e+1}(\cE)\ne 0$. We conclude that:
\begin{itemize}
\item the minimal number of sections needed to generate $\cE$ at all points is 
$$a_{min}(\cE):=e-1+\min \{ k\ge 0 \text{ such that } s_k(\cE)=0\};$$
\item any general subspace $A$ of $H^0(\PP^n,\cE)$, of dimension $a\ge a_{min}(\cE)$, generates 
$\cE$ at every point.
\end{itemize}
Similar considerations apply to $\cE^\vee(1)$. This means that we can, in a sense, focus on the 
case where $A=H^0(\PP^n,\cE)$ and $B=H^0(\PP^n,\cE^\vee(1))$. The corresponding matrix of linear forms can 
be described invariantly as given by the natural pairing defined by the composition
$$m_\cE : H^0(\PP^n,\cE)\otimes H^0(\PP^n,\cE^\vee(1))\ra 
H^0(\PP^n,\cE\otimes \cE^\vee(1))\ra 
H^0(\PP^n,\cO(1)).$$
If we choose basis $s_1,\ldots , s_N$ of $H^0(\PP^n,\cE)$ and $\sigma_1,\ldots, \sigma_M$
of $H^0(\PP^n,\cE^\vee(1))$, we get the matrix $M_\cE$ whose entry $(i,j)$  is the linear form 
$m_\cE(s_i,\sigma_j)$. This matrix has constant rank $e$. Moreover, it remains of constant 
rank $e$ when we take any $a\times b$ submatrix given by $a$ general linear combinations 
of rows and $b$ general linear combinations of columns, as long as $a\ge a_{min}(\cE)$ 
and $b\ge b_{min}(\cE):=a_{min}(\cE^\vee(1))$. For special linear combinations of rows and columns, or for $a< a_{min}(\cE)$ or  $b< b_{min}(\cE)$, the rank will no longer be constant 
but only bounded by $e$. 

Notice the complete symmetry between $\cF=\cE^\vee(1)$ and $\cE=\cF^\vee(1)$. Exchanging these two bundles, once basis are chosen for their spaces of global sections,  simply amounts to taking the transpose of the associated
matrix of linear forms. Since $c_1(\cE^\vee(1))=e-c_1(\cE)$, we can in particular
always suppose that $2c_1(\cE)\le e$, which will simplify the classifications.

\medskip\noindent {\it Remark.} Eisenbud and Harris proposed in \cite{MR954659} a specific
terminology for linear spaces of matrices of bounded (i.e. everywhere non-maximal) rank,
by distinguishing {\it decomposable and strongly indecomposable, primitive, unextendable, unliftable} and finally {\it basic} spaces. For the constant rank case that we discuss in this paper, we focus on the indecomposability of the associated bundle. It would 
certainly be interesting to understand how our results should be interpreted in the 
wider setting of spaces of matrices of bounded rank.

\subsection{Uniformity}

We have already mentioned the following important fact. We include a proof for the 
reader's convenience. Recall that a vector bundle $\cE$ 
is {\it uniform} if for every line $L\subset \PP^n$, the restriction
$$\cE_L\simeq \cO_L(m_1)\oplus \cdots\oplus\cO_L(m_e)$$
for some integers $m_1,\ldots, m_e$ that do not depend on $L$. We say $\cE$ is $1$-uniform 
if $0\le m_1,\ldots , m_e\le 1$. In this case the number of integers $k$ such that $m_k=1$ 
is equal to the first Chern class $c_1(\cE)$. 

\begin{prop}
    Suppose $\cE$ is a rank $e$ vector bundle on $\PP^n$, such that $\cE$ and $\cE^\vee(1)$
    are both generated by global sections. Then
    \begin{enumerate}
        \item $\cE$ is $1$-uniform;
        \item if $e\le n+1$ and $\cE$ is indecomposable, then $\cE=\cO_{\PP^n}, \cO_{\PP^n}(1)$,
        $Q$ or  $Q^\vee(1)$. 
    \end{enumerate}
\end{prop}

\proof (1) Being globally generated is preserved by restriction to a subvariety. So if $L$ is 
a line and $$\cE_L\simeq \cO_L(m_1(L))\oplus \cdots\oplus\cO_L(m_e(L)),$$ then $m_k(L)\ge 0$
for any $k$. Similarly, if $\cE^\vee(1)$ is also generated by global sections, we get that 
$m_k(L)\le 1$ for any $k$. Since the sum $m_1(L)+\cdots +m_e(L)=c_1(\cE)$ is fixed, the first 
assertion follows.  

(2) Uniform bundles of small rank have been classified \cite{VV, EHS, PhE}. A uniform vector bundle 
$\cE$ of rank $e\le n+1$ is a direct sum of line bundles and twists of the 
quotient bundle $Q$ and its dual.
If it is moreover $1$-uniform, the only possible simple factors are 
$\cO_{\PP^n}, \cO_{\PP^n}(1)$, $Q$ and $Q^\vee(1)$.\qed  

\begin{example}
When $\cE=Q$, let $V=H^0(\PP^n,\cO_{\PP^n}(1))^\vee$. Then 
$$H^0(\cE)=V, \qquad H^0(\cE^\vee(1))=\wedge^2V^\vee, $$ and $m_\cE$ is the natural pairing 
$$V\otimes \wedge^2V^\vee\ra V^\vee, \quad v\otimes \omega\mapsto \omega(v,\bullet).$$
This classical example was vastly generalized in \cite{landsberg-manivel-bounded}. 
\end{example} 

\subsection{Factoring out sections}
When $\cE$ is globally generated of rank $e>n$, a general section $s\in H^0(\PP^n,\cE)$ does not vanish anywhere
and therefore defines a vector bundle $\cF$ of rank $f=e-1$, with an exact sequence 
$$ 0\ra \cO_{\PP^n}\stackrel{s}{\rightarrow} \cE\ra \cF\ra 0.$$
In particular, $\cF$ is globally generated and  $H^0(\PP^n,\cF)=H^0(\PP^n,\cE)/\CC s$. When 
$\cE^\vee(1)$ is generated by global sections, can it also be the case of $\cF^\vee(1)$? 
We have a short exact sequence 
$$ 0\ra \cF^\vee(1)\ra \cE^\vee(1)\ra \cO_{\PP^n}(1)\ra 0,$$
showing that $H^0(\PP^n,\cF^\vee(1))\subset H^0(\PP^n,\cE^\vee(1))$ is the kernel of the morphism 
$$ m_\cE(s,\bullet) : H^0(\PP^n,\cE^\vee(1))\lra H^0(\PP^n,\cO_{\PP^n}(1))=V^\vee.$$
Suppose that the  morphism $m_\cE(s,\bullet)$ is surjective. In order to check whether $\cF^\vee(1)$ is generated by global sections, we can use the following diagram, where the horizontal and vertical short sequences are exact:

$$
\begin{CD}
    @. \cK_{\cE^\vee(1)} @>>> Q^\vee  \\
   @. @VVV        @VVV\\
  H^0(\cF^\vee(1))\otimes\cO_{\PP^n} @>>>  H^0(\cE^\vee(1))\otimes\cO_{\PP^n} @>>> V^\vee\otimes\cO_{\PP^n} \\
   @VVV        @VVV @VVV \\
  \cF^\vee(1) @>>> \cE^\vee(1)  @>>> \cO_{\PP^n}(1)
\end{CD}
$$

\medskip
We denoted by $\cK_{\cE^\vee(1)}$ the kernel of the evaluation morphism for $\cE^\vee(1)$. 
By the snake lemma, we deduce: 

\begin{prop}\label{factoring}
The evaluation morphism 
$$H^0(\cF^\vee(1))\otimes\cO_{\PP^n}
\ra \cF^\vee(1)$$ for $\cF^\vee(1)$
is surjective at $[v]$ if and only if the natural map 
$$V/\CC v \longrightarrow \cK_{\cE^\vee(1),v}^\vee, \qquad \overline{w}\mapsto (\sigma \mapsto m_\cE(s,\sigma)(w))$$
is injective. 
\end{prop}

In concrete terms, this means that we can choose a basis $s_1,\ldots , s_N$
of $H^0(\cE)$ with $s_1=s$, and a basis $\sigma_1, \ldots ,\sigma_M$ of $H^0(\cE^\vee(1))$ 
such that $m_\cE(s,\sigma_1), \ldots , m_\cE(s,\sigma_{n+1})$ form a basis of $H^0(\cO_{\PP^n}(1))$,
say $x_0,\ldots, x_n$, while $m_\cE(s,\sigma_k)=0$ for $k>n+1$, meaning that $\sigma_{n+2}, \ldots ,\sigma_M$ form a basis of  $H^0(\cF^\vee(1))$. The corresponding matrix $M_\cE$ 
is then of the form 

$$
M_\cE \;\;= \;\;\begin{pmatrix} x_0 & *\;\;&*&*&*\;\;& * \\x_1 & *\;\;&*&*&*\;\;&* \\
\cdots & *\;\;&*&*&*\;\;&* \\ \cdots & *\;\;&*&*&*\;\;&* \\x_n & *\;\;&*&*&*\;\;&* \\
0 & & & & & \\ 
\cdots & & & & & \\ 
\cdots & & & M_\cF  & & \\ 
\cdots & & & & & \\ 
0 & & & & & \end{pmatrix}
$$

\noindent where the matrix $M_\cF$ has constant rank $f=e-1$. We can apply 
this process several times if we want to factor out several sections of $\cE$, and the 
innocent looking criterion given by the previous proposition is easy to apply concretely. 

\medskip Conversely, the short exact sequence  $0\ra \cF^\vee(1)\ra \cE^\vee(1) \ra 
\cO_{\PP^n}(1)\ra 0$ indicates that we can expect that $\cE^\vee(1)$ is globally 
generated when $\cF^\vee(1)$ is. In other words, if $\cE^\vee(1)$ fails to be generated,
we cannot expect to correct this failure by extracting sections. It is not true
in general that extensions of globally generated vector bundles must be globally generated; 
but this can be checked to be true in some concrete situations, and the previous 
principle will apply. A typical case will be that of ordinary Drézet bundles. 

\smallskip In the next two sections we discuss two different types of vector
bundles, which have been studied a lot in the literature, but that we consider
under our specific point of view: Steiner bundles first, and then those we call
Drézet bundles, which are special instances of (duals of) the more general syzygy bundles. 

\section{Steiner bundles} 

Steiner bundles were defined in \cite{MR1240599} as vector bundles $\cE$ on
$\pp n$ arising from exact sequences
of the form 
$$
0\rightarrow \mathcal{O}_{\PP^n}(-1)^{\op c}\rightarrow \mathcal{O}_{\PP ^n}^{\oplus e+c} 
\rightarrow \cE\rightarrow 0.
$$

\subsection{First properties} 
It may be more natural to rewrite the previous exact sequence as 
$$
0\rightarrow \mathcal{O}_{\PP V}(-1)\otimes B_c\stackrel{M}{\rightarrow} \mathcal{O}_{\PP V}
\otimes W_{e+c} 
\rightarrow \cE\rightarrow 0,
$$
for a matrix $M$ defined by a tensor in $V^\vee\ot B_c^\vee\ot W_{e+c}$. 

\smallskip
Using the Beilinson spectral sequence, the property of being Steiner was characterized in \cite{MR1240599} by certain cohomological conditions; as a consequence, this is an open property. It was proved by Bohnhorst and Spindler (cited erroneously in \cite{MR1240599}) that any rank $n$ Steiner bundle on $\PP^n$ is stable (note that $n$ is the smallest possible rank of a Steiner bundle, since all its Chern classes are nonzero). 
This was extended in \cite[Theorem 1.7]{chs} to Steiner bundles such that   
 $n\le e<(n-1)c$ (actually under a slightly less restrictive condition).

\begin{remark} From the defining exact sequence of a Steiner bundle, we 
immediately get that 
its  Chern classes are  
$c_k=\binom{c+k-1}{k}$,  independent  of the rank. From these relations it is 
in general easy to check that a given vector bundle is not Steiner. 
\end{remark}
  
\medskip
Those Steiner bundles we are interested in are necessarily $1$-uniform, and as such have 
been studied in \cite{miroroig-marchesi}, where the following statement is proved:

\begin{prop}\label{miroroigmarchesi}
If a Steiner bundle $\cE$ is $1$-uniform and has no rank one summand, then 
$c+2n-2\le e\le cn$. All these values are possible, but for $e=cn$ the only possibility is 
$\cE=Q^{\oplus c}$.
\end{prop}

It is therefore quite natural to construct Steiner bundles of rank smaller than $cn$   by factoring out sections
from $Q^{\oplus c}$. The commutative diagram (where we omit zeroes)

$$\begin{CD}
    @. \cO_{\PP^n}(-1)\otimes B_{c} @=  \cO_{\PP^n}(-1)\otimes B_c \\
   @.  @VVV        @VVV\\
\cO_{\PP^n}\otimes A_k   @>>> \cO_{\PP^n}\otimes V\otimes B_c @>>>
\cO_{\PP^n}\otimes W_{e+c} \\
   @|  @VVV        @VVV\\
\cO_{\PP^n}\otimes A_k   @>>> Q \otimes B_c @>>> \cE 
\end{CD}$$

\medskip\noindent exchanges the subspace $A_k$ of 
$V\otimes B_c$ with subspace $B_c$ of $V^\vee\otimes W_{e+c}$, where $e=cn-k$. This is an instance of the classical {\it castling transform}. 

\subsection{Global generation}
From our perspective, the main problem is to understand when 
$\cE^\vee(1)$ is globally generated. 
For the general Steiner bundle, an essentially complete answer to this question 
was obtained  in \cite{MR1174900}.



\begin{prop} \label{prop31}
Suppose that $e\ge \frac{n+1}{2}c+1$. If $c\le 3$ suppose moreover that $e\ge n$ for $c=1$, 
$e\ge 2n$ for $c=2$, $e\ge 2n+1$ for $c=3$ (and even $e\ge 6$ for $c=3$ 
and $n=2$). 

Then for the generic Steiner bundle $\cE$ of rank $e$ on $\PP^n$ with $c=c_1(\cE)$,
the twisted dual $\cE^\vee(1)$ is generated by global sections. 
\end{prop}

Moreover $\cE^\vee(1)$ is acyclic and its sections are given by 
either one of the exact sequences 
$$0\ra H^0(\cE^\vee(1))\ra W_{e+c}^\vee\otimes V^\vee\ra B_c^\vee\otimes S^2V^\vee\ra 0,$$
$$0\ra H^0(\cE^\vee(1))\ra B_{c}^\vee\otimes \wedge^2V^\vee\ra A_k^\vee\otimes V^\vee\ra 0.$$

\begin{remark} A very similar situation, but with vector bundles $\cE^\vee(1)$ which do not have enough sections to be generated, was discussed in \cite{MR1717579} and   provides   examples of spaces of  matrices of bounded   rank.  
 \end{remark}

\begin{example}\label{steiner-drezet}
Suppose that $n=2$, $e=5$, $c=3$. By castling transform, 
$\cE$ is defined by an exact sequence 
$$0\ra \cO_{\PP^2}  \ra  Q \otimes B_3 \ra \cE \ra 0$$
where the first morphism is given by a tensor in $V\otimes B_3$. In general 
this tensor has maximal rank and identifies $B_3$ with $V^\vee$, and 
then $H^0(\cE)$ with $\fsl(V)$. It follows that $\cE$ is in fact homogeneous, 
and that $H^0(\cE^\vee(1))$ can be identified with $S^2V\otimes \det(V)^\vee$. 
Moreover $\cE^\vee(1)$ is globally generated although it has only $6$ sections, and one can check that $s_2(\cE^\vee(1))=0$ (in fact $\cE^\vee(1)$ is a Steiner bundle). So we get a matrix of constant rank, which can be interpreted as the 
equivariant morphism $V\ra \thom(S_2V,S_{21}V)$. This is one of the toy examples
generalized in \cite{landsberg-manivel-bounded}.
 \end{example}
 
\medskip
For $c=1$, the condition $e\ge n$ is already necessary for the mere existence of a Steiner bundles, as we already mentioned. For $c=2$, the condition that $e\ge 2n$ is necessary, in 
view of the following (see also Proposition \ref{miroroigmarchesi}). 

\begin{example}\label{drops}
Consider the direct sum of two copies of the quotient 
bundle $Q$ and let us try to factor out a general section $s\in H^0(Q\oplus Q)=V\oplus V$. So 
$s=(s_1,s_2)$ for two independent vectors $s_1, s_2\in V$. Applying Proposition \ref{factoring}, 
we see that the corresponding bundle $\cE^\vee(1)$ is generated by global sections at $[v]$
if and only if the natural morphism 
$$ V/\CC v \ra \wedge^2(V/\CC v)\oplus \wedge^2(V/\CC v), \quad 
\overline{w}\mapsto 
(\overline{w}\wedge \overline{s_1}, \overline{w}\wedge \overline{s_2})$$
is injective. But this fails! Indeed it clearly fails when $v=s_1$ since $\overline{s_2}$ is sent  to zero. In fact it fails 
exactly on the projective line joining $[s_1]$ to $[s_2]$, and $\cE^\vee(1)$ fails to be globally
generated exactly on this line. 
\end{example}

\medskip For $c=3$ this  problem does not happen. 

\begin{example}\label{3Q-1}
Consider the direct sum of three copies of the quotient 
bundle $Q$, and let us try to factor out a general section $s\in H^0(Q\oplus Q\oplus Q)=
V\oplus V\oplus V$. So 
$s=(s_1,s_2,s_3)$ for three independent vectors $s_1, s_2, s_3\in V$.
Then the corresponding vector bundle $\cE^\vee(1)$ is generated by global sections at $[v]$
if and only if the natural morphism 
$$ V/\CC v \ra \wedge^2(V/\CC v)^{\oplus 3}, \quad \overline{w}\mapsto 
(\overline{w}\wedge \overline{s_1}, \overline{w}\wedge \overline{s_2}, \overline{w}\wedge \overline{s_3})$$
is injective. Observe that $\overline{w}\wedge \overline{s_1}$ is zero 
exactly when 
$w\in \langle v,s_1\rangle$. So if also $\overline{w}\wedge \overline{s_2}=0$ and 
$\overline{w}\wedge \overline{s_3}=0$, then 
$$w\in \langle v,s_1\rangle \cap \langle v,s_2\rangle\cap \langle v,s_3\rangle=\CC v$$
for any $v$ since  $s_1, s_2, s_3$ are independent. So in this case $\cE^\vee(1)$ is
generated by global sections, with
$$H^0(\cE^\vee(1))=\Big\{(\omega_1,\omega_2, \omega_3)\in\wedge^2V^\vee, \;\; 
\omega_1(s_1,\bullet)+\omega_2(s_2,\bullet)+\omega_3(s_3,\bullet)=0\Big\}.$$
\end{example}

\medskip A natural question is whether extracting sections from a decomposable vector bundle,
we can get an indecomposable vector  bundle. Consider once again an exact sequence  
$$0\ra A_k\otimes \cO_{\PP^n}\ra B_c\otimes Q\ra \cE\ra 0$$
defined by a tensor $M\in A_k^\vee\otimes B_c\otimes V$. Note that $\cE$ is a vector 
bundle if and only if the image of the induced map $A_k\ra B_c\otimes V$ does not 
contain any rank one tensor. Dualizing, we get $H^q(\cE^\vee)=\delta_{q,1}A_k^\vee$,
and after tensoring by $Q$ and taking cohomology, we obtain the exact sequence 
$$0\ra H^0(\cE^\vee\otimes Q)\ra B_c^\vee\ra A_k^\vee\otimes V \ra 
H^1(\cE^\vee\otimes Q)\ra 0.$$
If we suppose the morphism $B_c^\vee\ra A_k^\vee\otimes V$ to be injective,
we deduce the exact sequence
$$0\ra H^0(End(\cE))\ra End(A_k)\ra B_c\otimes (A_k^\vee\otimes V)/B_c^\vee \ra 
H^1(End(\cE))\ra 0.$$
Note that $B_c\otimes (A_k^\vee\otimes V)/B_c^\vee$ is the tangent space at 
$B_c^\vee$ to the Grassmannian $Gr(c, A_k^\vee\otimes V)$. We can see this 
Grassmannian as the quotient of $A_k^\vee\otimes B_c\otimes V$ by $GL(B_c)$. 
Hence the following statement:

\begin{prop} Let $G$ denote the image of the group $GL(A_k)\times GL(B_c)$ inside 
$GL(A_k^\vee\otimes B_c)$. Suppose that $M\in A_k^\vee\otimes B_c\otimes V$ has finite stabilizer in $G$, and that the induced maps $A_k\ra B_c\otimes V$ and $B_c^\vee\ra A_k^\vee\otimes V$ are both
injective.

Then $\cE$ is simple, and in particular indecomposable. \end{prop}

A more precise simplicity criterion is given in \cite{MR2182426}. Under our conditions we can also conclude that 
the moduli space or stack parametrizing deformations of $\cE$ is made of 
vector bundles of the same type, and  is essentially the quotient of 
$A_k^\vee\otimes B_c\otimes V$ by $G\simeq GL(A_k)\times GL(B_c)/ \CC^*$. 
Up to the action of $GL(V)$  we get the quotient of  $A_k^\vee\otimes B_c\otimes V$ by $GL(A_k)\times GL(B_c)\times GL(V)$, which has been much studied. 


\begin{remark}
Finiteness of the generic stabilizer is to be expected in general, and in fact 
it holds as soon as it is compatible with dimensions. A general classification of irreducible representations of semisimple complex Lie groups with
generic stabilizer of positive dimension was obtained in \cite{elashvili}.
\end{remark}


\begin{example}
A well-known case is $k=2$, $c=n=3$, for which we get the variety of twisted
cubics in $\PP^3$, as a parameter space for a family of rank seven vector bundles. 
Moreover the action of $PGL_4$ on this quotient is quasi-homogeneous, with an open subset corresponding to the smooth twisted cubics; hence an essentially unique matrix of constant 
rank 7 that we are going to compute. The starting point is the presentation of the 
ideal sheaf of the twisted cubic, as being generated by the $2\times 2$ minors of
$$\begin{pmatrix} x_1& x_2&x_3 \\ x_2&x_3&x_4\end{pmatrix}.$$
The corresponding tensor in $A_2^\vee\otimes B_3\otimes V$ is
$$M=\alpha_1\otimes (b_1\otimes c_1+b_2\otimes c_2+b_3\otimes c_3)+
\alpha_2\otimes (b_1\otimes c_2+b_2\otimes c_3+b_3\otimes c_4).
$$
This means that $H^0(\cE)$ can be identified with the quotient of $B_3\otimes V$ by the 
pencil generated by $b_1\otimes c_1+b_2\otimes c_2+b_3\otimes c_3$ and $
b_1\otimes c_2+b_2\otimes c_3+b_3\otimes c_4$. We choose for basis 
of this $10$-dimensional space the classes of $b_1\otimes c_3, b_1\otimes c_4, 
b_2\otimes c_1, \ldots , b_3\otimes c_4$. 

On the other hand, $H^0(\cE^\vee(1))$ is the kernel of the contraction by $M$ from 
$B_3^\vee\otimes \wedge^2V^\vee$ to $A_2^\vee\otimes V^\vee$. We check that this kernel 
has the expected dimension $10$, with a basis given by the tensors
$$\beta_1\otimes x_{34}, \beta_2\otimes x_{14}, \beta_3\otimes x_{12}, \beta_2\otimes x_{12}-\beta_3\otimes x_{13}, \beta_1\otimes x_{24}-\beta_2\otimes x_{34}, $$
$$\beta_1\otimes x_{13}-\beta_2\otimes x_{23}+\beta_3\otimes x_{24}, 
\beta_2\otimes x_{13}-\beta_3\otimes x_{14}, \beta_1\otimes x_{14}-\beta_2\otimes x_{24}, $$
$$\beta_1\otimes x_{12}-\beta_3\otimes (x_{14}- x_{23}),\; \beta_1\otimes (x_{14}- x_{23})
-\beta_3\otimes x_{34}.
$$
The resulting $10\times 10$ matrix of constant rank $7$ is the following:

$$\begin{pmatrix}
 x_4&-x_3&0&0&0&0&0&0&0&0 \\
 0&0&x_4&0&0&-x_1&0&0&0&0 \\
 0&0&0&0&0&0&x_2&-x_1&0&0 \\
 0&0&x_2&-x_1&0&0&-x_3&0&x_1&0 \\
 0&-x_2&0&0&x_4&-x_3&0&0&0&0 \\
 -x_1&0&0&x_3&-x_2&0&0&x_4&0&-x_2 \\
 0&0&x_3&0&-x_1&0&-x_4&0&0&x_1 \\
 0&-x_1&0&-x_4&0&x_2&0&0&0&0 \\
 0&0&0&0&0&0&-x_4&x_3&-x_2&x_1 \\
 x_2&-x_1&0&0&0&0&0&0&x_4&-x_3 
\end{pmatrix}.$$

\end{example}

\section{Drézet bundles}

In this section we discuss a class of vector bundles that were used by Drézet to construct 
examples of uniform bundles of small rank which are not homogeneous; namely, rank $2n$ non-homogeneous uniform vector bundles on $\PP^n$ for any $n\ge 2$ \cite{drezet-uniform}. 

\subsection{Definition and first examples} Our terminology is the following. 

\medskip \noindent {\it Definition.} A {\it Drézet bundle}  on $\PP^n$ is 
a vector bundle $\cE$ fitting into an exact sequence 
$$0\ra \cO_{\PP^n}(-c)\ra  \cO_{\PP^n}^{\oplus e+1}\ra\cE\ra 0.$$
Of course $c=c_1(\cE)$ and $s(\cE)=1-ch$. 

\medskip The dual vector bundle  $\cE^\vee$ is a 
special case  of what is called a {\it syzygy bundle}.
The stability of syzygy bundles was discussed in \cite{Marques-MiroRoig}. It follows from 
\cite[Theorem 4.6]{Marques-MiroRoig} that a general Drézet bundle is stable when $n\le e\le \binom{n+c}{c}$. 

\begin{example}\label{pfaffian}
Consider the space $\cA_m$ of skew-symmetric matrices of 
size $m$. If $m=2p+1$ is odd, the generic matrix in $\cA_m$ has rank $2p$, hence a one 
dimensional kernel given by its $p$-th power as a skew-symmetric two-form. Morever the closed 
subset of matrices of rank smaller than $2p$ has codimension three. 
Taking a general plane 
$P\simeq\PP^2$  inside $\PP(\cA_m)$, we thus get a Drézet bundle 
$$0\ra \cO_{P}(-p)\ra  \cO_{P}^{\oplus 2p+1}\ra\cE_P\ra 0.$$
Hence a family of vector bundles on $\PP^2$ parametrized by $Gr(3,\cA_m)$. In this case 
$\cE_P^\vee(1)\cong \cE_P$, in particular it is generated by sections \cite{MR0453723}. 
This self-duality 
property implies that  $\cE_P$ is  uniform of splitting type $(0^p,1^p)$. In the 
sequel we will call these bundles {\it Pfaffian}, since $P$ is generated by the 
Pfaffians of the  $2p\times 2p$
skew-symmetric minors of the skew-symmetric matrix $\cA_{2p+1}$. Note that Pfaffian bundles only exist on $\PP^2$.
\end{example}

\begin{example}\label{universalDrezet} 
The universal Drézet bundle on $\PP^n$ with first Chern class $c$ is the homogeneous bundle defined by the sequence 
$$0\ra \cO_{\PP^n}(-c)\ra  S^cV\otimes \cO_{\PP^n}\ra\cE_c\ra 0,$$
where $V=\CC^{n+1}$ and $\PP^n=\PP(V)$. This homogeneous bundle is discussed in 
\cite{landsberg-manivel-bounded}. Any Drézet bundle of higher rank with the 
same first Chern class is the direct sum of $\cE_c$ with a trivial factor. Any Drézet 
bundle of smaller rank is obtained by factoring out sections of $\cE_c$. 
\end{example}

\subsection{Deformations}
A Drézet bundle $\cE$ of rank $e$ with first Chern class $c$ is defined by a collection $P_0,\ldots, P_e$ of degree $c$ polynomials, or rather by the subspace $P=\langle P_0,
\ldots , P_e\rangle$ they generate in $S^cV^\vee$. We will always suppose $P$ to be 
of dimension $e+1$, since otherwise $\cE$ has a trivial factor. We thus consider the 
family of Drézet bundles as parametrized by the open subset of $Gr(e+1,S^cV^\vee)$
parametrizing spaces of polynomials with no common zero, and we write their 
defining sequence as 
$$0\ra \cO_{\PP^n}(-c)\ra P^\vee\otimes \cO_{\PP^n}\ra\cE_P\ra 0.$$

After dualizing, we get $H^1(\cE^\vee_P)=S^cV^\vee/P$, while $H^q(\cE^\vee_P)=0$ for $q\ne 1$.
Twisting by $\cO_{\PP^n}(-c)$ and supposing $c\le n$, we get $H^1(\cE_P^\vee(-c))=\CC$
and $H^q(\cE_P^\vee(-c))=0$ for $q\ne 1$. The short sequence 
$$0\ra \cE_P^\vee(-c)\ra P^\vee\otimes \cE_P^\vee\ra\cE_P^\vee\otimes \cE_P\ra 0$$
then yields the following  consequences:

\begin{prop}
Any Drézet bundle $\cE_P$ on $\PP^n$ with fist Chern class 
$c_1(\cE_P)= c\le n$ is simple. Moreover 
$$H^1(End(\cE_P))=Hom(P, S^cV^\vee/P)=T_{[P]}Gr(e+1,S^cV^\vee),$$
and $H^q(End(\cE_P))=0$ for any $q>1$. 
\end{prop}

In particular,  for $c\le n$  local deformations of Drézet bundles are again Drézet bundles, and 
we get a smooth family of vector bundles parametrized by an open subset of the Grassmannian
$Gr(e+1,S^cV^\vee)$.  For $n=2$ and $c>2$ 
it is no longer true that local deformations of Drézet bundles are again Drézet bundles, see \cite[Theorem 4.4]{Costa-Macias-MiroRoig}.

\subsection{Sections and global generation of the twisted dual} From the short exact sequence that defines $\cE_P$, after dualizing, twisting and taking cohomology we get the exact 
sequence 
$$0\ra H^0(\cE_P^\vee(1))\ra P\otimes V^\vee\ra S^{c+1}V^\vee \ra H^1(\cE_P^\vee(1))\ra 0.$$
An obvious consequence is:

\begin{prop}\label{syz}
$H^0(\cE_P^\vee(1))\simeq Syz_1(P)$, the space of linear syzygies between polynomials in $P$. 
\end{prop}

\medskip \noindent {\it Definition.} The {\it Drézet bundle} $\cE_P$ on $\PP^n$ is 
{\it ordinary} if $H^1(\cE_P^\vee(1))=0$. This is equivalent to the condition that 
$$\forall C\in S^{c+1}V^\vee, \exists \ell_0,\ldots, \ell_e\in V^\vee \qquad C=\ell_0P_0+
\cdots +\ell_e P_e.$$
Of course this is only possible when $e$ is large enough. More precisely, it immediately follows from \cite{Hochster-Laksov} that  when
$$(e+1)(n+1)\ge \binom{n+c+1}{c+1},$$
a general Drézet bundle is ordinary.
In addition,  If $e$ is too
small for this condition to be fulfilled, then a general Drézet bundle verifies 
$H^0(\cE_P^\vee(1))=0$.

\medskip For ordinary Drézet bundles, we can use the snake lemma to reformulate the 
condition that $\cE_P^\vee(1)$ is globally generated at $[v]$ into the following condition:
$$\forall C\in S^{c+1}V^\vee, C(v)=0, \;\; \exists \ell_0,\ldots, \ell_e\in V^\vee, \hspace*{3cm}$$
$$\hspace*{3cm}\ell_0(v)=\cdots= \ell_e(v)=0,
\hspace*{1cm} C=\ell_0P_0+
\cdots +\ell_e P_e.$$
Dualizing, we get the following equivalent condition:

\begin{lemma}\label{criterion}
    Suppose the Drézet bundle $\cE_P$ is ordinary. Then $\cE_P^\vee(1)$ 
    is globally generated at $[v]$ if and only if the map
$$S^{c+1}V/\CC v^{c+1}\ra (V/\CC v)^{e+1}, \qquad T\mapsto (P_0\rfloor T,
\ldots , P_e\rfloor T)$$
is injective.
\end{lemma}

Here $Q\rfloor T$ is just the vector obtained by contracting the tensor 
$T$ with the polynomial $Q$. 
Note the obvious consequence: 

\begin{coro}\label{inheritance} 
Suppose that $P'\subset P\subset S^cV^\vee$. If $\cE_{P'}$ is an ordinary Dr\'ezet bundle and $\cE_{P'}^\vee(1)$ is generated by global sections, then $\cE_{P}$ is also an ordinary Dr\'ezet bundle and $\cE_{P}^\vee(1)$ is also generated by global sections.
\end{coro}

Applying Lemma \ref{criterion} we can prove that a generic Dr\'ezet bundle of 
high enough rank always gives rise to a matrix of linear forms of constant rank. 

\begin{theorem}\label{drezet-generation}
Suppose that the rank $e$ is large enough, more precisely that 
$$\dim P > \Bigg(\frac{1}{c}+\frac{1}{n+1} \Bigg)\dim S^cV^\vee.$$
Then the general Drézet bundle $\cE_P$ is ordinary, and 
$\cE_P^\vee(1)$ is generated by global sections. 
\end{theorem}

\proof First observe that the map $\theta_T$ from $S^cV^\vee$ to $V$ 
that sends $Q$ to $Q\rfloor T$
is in general surjective. Indeed, if it is not, suppose its image is contained in a 
hyperplane $H=\langle e_1,\ldots , e_n\rangle$ of $V=\langle e_0,\ldots , e_n\rangle$; then 
the tensor $T$, when written in this basis, can have only zero coefficient on any monomial
$e_0M$ involving $e_0$; otherwise, contracting with the monomial  $M^\vee$ in the dual 
coordinates we would get a vector with a nonzero coefficient on $e_0$, a contradiction. 
In other words $T$ has to belong to $S^{c+1}H$. More generally, if $W_T\subset V$
 is the image of $\theta_T$, then $T$ has to belong to $S^{c+1}W_T$, and $W_T$ is minimal for this property. Let us call the dimension of $W_T$ the {\it global rank} of $T$.

 Then consider the set $I$ where global generation fails, that is the set
 $$I:= \Big\{ (P,[v],[T])\in Gr(e+1,S^cV^\vee)\times \PP(V)\times \PP(S^{c+1}V),
 \hspace*{2cm}$$
 $$\hspace*{5cm} T\notin \CC v^{c+1}, \;\; 
 P_0\rfloor T, \ldots , P_e\rfloor T\in \CC v \Big\}.$$
Let us show that the dimension of $I$ is strictly smaller than the dimension of 
$Gr(e+1,S^cV^\vee)$. For this we project to $\PP(V)\times \PP(S^{c+1}V)$. 
If $([v],[T])$ is in the image, then necessarily $v\in W_T$. Denote the dimension of $W_T$ 
by $d$. 
Then $P$ must be a subspace of $\theta_T^{-1}(\CC v)$, which has codimension $d-1$.
Note that if $d=1$, $T$ has to be a multiple of some $t^{c+1}$ with $t$ independent 
of $v$, and then $P_0\rfloor T, \ldots , P_e\rfloor T$ are multiples of $t$; since
they are also multiples of $v$, they must in fact be zero, which means that $P_0(t)=\cdots =P_e(t)$. In other words $P$ has a common zero and cannot be generic if $e\ge n$. So we can exclude this 
case and bound the 
dimension of $I$ by the maximum, for $2\le d\le n+1$,  of 
$$
\dim Gr(d,n+1)+\dim S^{c+1}\CC^d+\dim Gr(e+1,\dim S^{c}\CC^{n+1}-(d-1)).$$
Let us rewrite this quantity as $\dim Gr(e+1,\dim S^{c}\CC^{n+1})-N_d$
where 
$$N_d=(d-1)(e+1)-d(n+1-d)-\binom{c+d}{c+1}.$$
We will therefore be able to conclude that for a general $P\in Gr(e+1,\dim S^{c}\CC^{n+1})$,
there is no $([v],[T])$ contradicting the global generation of $\cE_P^\vee(1)$ as soon
as $N_d>0$ for all $d$. Note that 
$$N_d-N_{d+1}=n-e+\binom{c+d}{c}$$
implies that $N_d$ is convex as a function of $d$.  So we only need to check that
$N_2$ and $N_{n+1}$ are strictly positive, which amounts to
$$ e\ge 2n+c \quad \mathrm{and}\quad e+1>\frac{1}{n}\binom{c+n+1}{c+1}.$$
Since $\dim S^cV^\vee=\binom{c+n}{c}=\frac{c+1}{c+n+n}\binom{c+n+1}{c+1}$, our claim follows.\qed

\medskip\noindent {\it Remark}.
In case the  Drézet bundle  is not ordinary, a substitute to Lemma \ref{criterion} is the following criterion for global generation:

\begin{lemma}\label{criterion-bis}
The Drézet bundle $\cE_P$ is generated by global sections at $[v]$ if and only if,
for any $q_0\in P$ such that $q_0(v)=0$, one has $q_0 V^\vee\subset P v^\perp$
in $S^cV^\vee$. 
\end{lemma}

\begin{example}\label{drezet-steiner}
On $\PP^2=\PP(V_3)$, the universal  Drézet bundle $\cE_2$ 
with first Chern class $c=2$ has rank $5$. It is an extension of $S^2Q$ by 
$Q(-1)$ and $H^0(\cE_2^\vee(1))\simeq \fsl_3$. This vector bundle gives rise to the
following $6\times 8$
matrix of constant rank five  \cite[Proposition 11]{landsberg-manivel-bounded}:
$$\begin{pmatrix}
-x_1 & -x_2 & 0 & 0  &0 &0 &0 &0   \\ 
0  &0 &x_0 &0 &-x_2 &0 & 0& 0 \\
0   &0 & 0 & x_0& 0& x_1& 0& 0  \\
x_0    &0 &-x_1 &0 &0 &0   &-x_2 & x_2 \\
0     &x_0 &0 & -x_2& 0&0   & 0& -x_1\\
0      &0 &0 &0   &x_1 &-x_2 &x_0 &0  
\end{pmatrix}.$$
This is in fact the twisted dual of the Steiner bundle discussed in Example~\ref{steiner-drezet}.

We can factor out a generic section by restricting to the hyperplane of 
quadrics orthogonal to the tensor $t=e_0e_1+e_2^2$, which is 
$$H=\langle q_1= x_0x_1-x_2^2, q_2=x_0x_2, q_3=x_1x_2, q_4=x_0^2, q_5=x_1^2\rangle.$$
Note that $t$ is invariant under the action of a copy of the orthogonal group 
in three variables, which has only two orbits in $\PP^2$. Therefore,  to prove that 
$\cE_H^\vee(1)$ is globally generated, it is enough to check it at one point 
of the closed orbit, say $[e_0]$. So consider a cubic symmetric tensor 
$T=\sum_{ijk}T_{ijk}e_ie_je_k$, and suppose that the contractions with all quadrics 
in $H$ give multiples of $e_0$. Applying this successively to $q_2,q_3,q_4,q_5$
yields $T_{012}=T_{022}=0$, $T_{112}=T_{122}=0$, $T_{001}=T_{002}=0$ and $T_{111}=T_{112}=0$. 
There only remains $T=T_{000}e_0^3+3T_{011}e_0e_1^2+T_{222}e_2^3$, whose contraction with the remaining quadric $q_1$ is $3T_{011}e_1-T_{222}e_2$. This is a multiple of $e_0$ if and only if $=T_{000}e_0^3$, and we conclude that $\cE_H^\vee(1)$ is generated by global sections. 
It has five independent sections given by the five syzygies 
$$x_0q_1-x_1q_4+x_2q_2, x_1q_1-x_0q_5+x_2q_3, x_0q_2-x_2q_4, x_1q_1-x_0q_3, x_1q_3-x_2q_5,$$
hence the following $5\times 5$ matrix of constant rank four:
$$\begin{pmatrix}
x_0 & x_2 & 0 & -x_1 &0   \\ 
x_1  &0 &x_2 &0 &-x_0  \\
0   &x_0 & 0 & -x_2& 0  \\
0 &x_1 &-x_0 &0 &0   \\
0 &0    &x_1 &0 & -x_2 
\end{pmatrix}.$$

Note that if we replace $t$ by some rank two tensor $t'=e_0e_1$, any quadric in its 
orthogonal hyperplane $H'$ contracts the tensor $T=e_0e_1^2$ to a multiple of $e_0$, 
showing that $\cE_{H'}^\vee(1)$ is not globally generated at $[e_0]$. 

If we degenerate further to a rank one tensor $t''=e_0^2$, then the orthogonal hyperplane
$H''\subset S^2V^\vee$ is such that $H''V^\vee$ is only a hyperplane in $S^3V^\vee$, so that 
$\cE_{H''}$ is not ordinary: it has $h^1(\cE_{H''}^\vee(1))=1$ and $h^0(\cE_{H''}^\vee(1))=6$. 
We get in this case a $6\times 5$ matrix of rank at most four:
$$\begin{pmatrix}
x_0 & 0 &-x_1 & 0 &0   \\ 
x_2 & 0  &0 &-x_1 &0  \\
0   &x_0 & 0 &0 & -x_2  \\
0 &x_1 &0&-x_2  &0   \\
0 &0    &x_2 &-x_0&0\\
0&0& x_2& 0&-x_1 
\end{pmatrix}.$$
The minor corresponding to rows $2345$ (resp. $1246$) and columns $1345$ (resp. $2345$)
is $x_2^4$  (resp. $x_1^4$), showing that the rank can drop only at $[e_0]$, where it is 
indeed equal to three. In particular, although $\cE_{H''}(1)$ has extra global sections, it is 
not globally generated. 
\end{example}

\section{Classification for $c_1=2$ and matrices of constant rank four and five}

Globally generated vector bundles on projective spaces
were classified for $c_1\le 5$ \cite{su1, su2, anghel-manolache, anghel-coanda-manolache}. We need to determine inside this classification which vector  bundles 
$\cE$ are such that $\cE^\vee(1)$ is also generated by global sections. Recall that this implies that $\cE$ is uniform, so we can suppose that the rank of $\cE$ is
$e\ge n+2$; otherwise $\cE$ is known to be a direct sum
of irreducible homogeneous bundles (line bundles, the quotient bundle and its twisted dual). 

We can also suppose that $\cE$ is indecomposable. In particular, since  $\cE^{\vee}(1)$ is
globally generated and has no trivial factor, $H^0(\cE(-1))$ must vanish. Together with the
rank condition $e\ge n+2$, this drastically simplifies the classification in case the 
first Chern class $c=c_1(\cE)\le 3$, and only a handful of cases have to be discussed. 
Note that there is also a partial classification for $c=4$ and even $c=5$, but only 
under the hypothesis that $h^1(\cE^\vee)=0$, which for our matters is unrealistic.

For $c=1$, the fact that $\cE$ can only be $\cO_{\PP^n}(1)$ or the quotient bundle $Q$ 
already appears  in \cite[Theorem 2.4]{MR954659}. In section 4 of the same paper, 
Eisenbud and Harris suggest to   focus (for constant and also bounded rank) 
on the case $c_1=2$, which we shall discuss now. 

\subsection{Vector bundles with $c_1=2$} 

According to \cite[Theorem 1.1]{su1}, there are only two possibilities.

\begin{prop}\label{classification-c=2} 
An indecomposable, globally generated vector bundle $\cE$ on $\PP^n$ 
of rank $e\ge n+2$, with $c_1(\cE)= 2$,  must be either a Steiner bundle or a Dr\'ezet bundle. 
\end{prop}

\subsubsection{The Steiner case}

Steiner bundles with $c_1=2$ are defined by an exact sequence of the form
$$0\ra \cO_{\PP^n}(-1)\oplus \cO_{\PP^n}(-1) \ra  \cO_{\PP^n}^{\oplus e+2}\ra \cE\ra 0.$$
Observe by dualizing this sequence that $H^0(\cE^\vee)$ has to be nonzero as soon as 
$e>2n$, and then $\cE$ admits a trivial factor. If $e=2n$ and $H^0(\cE^\vee)=0$ then 
necessarily $\cE=Q\oplus Q$ is again decomposable. %

\begin{prop}\label{steiner-c=2}
    There is no indecomposable Steiner bundle $\cE$ such that $c_1(\cE)=2$ and 
    $\cE^\vee(1)$ is generated by global sections. 
    \end{prop}

    \proof This is an immediate consequence of \cite[Theorem 13]{miroroig-marchesi},
    where it is proved that if $\cE$ is a $1$-uniform Steiner bundle on $\PP^n$ bundle of rank $e$
    with no trivial factor, then $c+2n-2\le e\le cn$ where $c=c_1(\cE)$. This applies in 
    our setting and for $c=2$ yields $e=2n$, in which case we know that $\cE=Q\oplus Q$
    is decomposable.\qed

\subsubsection{The Drézet case}
As a consequence of Theorem \ref{drezet-generation} and its proof, a generic Drezet bundle 
$\cE$ of rank 
$e>\frac{n^2+6n+11}{6}$ is ordinary, and $\cE^\vee(1)$ is generated by global sections,
which are given by linear syzygies between the quadrics defining $\cE$.
The maximal case is that of the homogeneous bundle $\cE_2$, of rank $e=\frac{n(n+3)}{2}$. 

\smallskip
Now consider  $\cE_P$ for $P$ a hyperplane in $S^2V^\vee$. Its orthogonal in $S^2V$ is the 
line generated by some tensor $t$ of rank $k$. Note that $\cE_P$ is a vector bundle only if 
$k>1$. (In general, $\cE_P$ is a vector bundle exactly 
when $P^\perp$ contains no rank one tensor.)

\begin{prop}\label{corank1}
 For any $k\ge 2$,  $\cE_P$ is an ordinary Dr\'ezet bundle.
 Moreover $\cE_P^\vee(1)$ is generated by global sections if and only if $k>2$.
\end{prop}

\proof Let us choose a basis of $V$ such that $t=e_0e_1+e_2^2+\cdots +e_{k-1}^2$. Its 
orthogonal $P$ is generated by the quadratic 
forms $x_0^2, x_1^2, x_0x_1-\frac{1}{2}x_2^2, x_2^2-x_3^2, \ldots , 
x_{k-2}^2-x_{k-1}^2, x_k^2,\ldots, x_n^2$,and the $x_ix_j$ for $i<j$ and $(i,j)\ne (0,1).$

In order to check whether $\cE_P^\vee(1)$ is generated by global sections  we can use
the fact that $t$ has a big stabilizer, whose orbits in $\PP(V)$ are: the support 
$\PP(S_t)=\PP(\langle e_0,\ldots , e_{k-1}\rangle)$ of $t$, the quadric $Q_t\subset \PP(S_t)$ 
defined by $t$, and the complement of $\PP(S_t)$. Since the locus where $\cE_P^\vee(1)$ is generated by global sections is closed and preserved by the stabilizer of $t$, we just need
to check whether  $\cE_P^\vee(1)$ is generated or not at one given point of $Q_t$, say $[e_0]$. 

So we consider a cubic tensor $C\in S^3V$, such that $Q\rfloor C$ is a multiple of $e_0$ 
for any $Q\in P$, and we must decide whether $C$ must be a multiple of $e_0^3$. 
So we may suppose that $C=e_0^2\lambda+e_0\kappa+\gamma$, where $\lambda, \kappa, \gamma$
do not depend on $e_0$. That $x_0^2\rfloor C$ is a multiple of $e_0$ implies $\lambda=0$. 
 That $x_ix_j\rfloor C$ is a multiple of $e_0$ for any $0<i<j$ implies that $\gamma_{ijk}=0$
 for any triple $ijk$ with a least two distinct indices; so we can write $\gamma=\sum_{i>0}
 \gamma_ie_i^3$, and then the fact  that  $Q\rfloor \gamma$ must be zero for $Q=x_1^2, 
x_2^2-x_3^2, \ldots , x_{k-2}^2-x_{k-1}^2, x_k^2,\ldots, x_n^2$ implies that $\gamma=0$. 
So we remain with $C=e_0\kappa$, and the condition that $x_0x_i\rfloor C$ is a multiple 
of $e_0$ for any $i>1$ implies that $x_i\rfloor \kappa=0$. Since this is true for any $i>1$ 
we only remain with $\kappa=\kappa_1e_1^2$, hence $C=\kappa_1e_0e_1^2$. 

If $k>2$, we can use the extra condition that $(x_0x_1-\frac{1}{2}x_2^2)\rfloor C$ is a multiple of $e_0$ to conclude that $\kappa_1=0$, and therefore that $\cE_P^\vee(1)$ is generated at $[e_0]$.
But if $k=2$, this extra condition is not available and we arrive at the opposite conclusion: the cubic $C=e_0e_1^2$ shows  that $\cE_P^\vee(1)$ is not generated at $[e_0]$.
\qed 

\medskip We could use the classification of pencils and webs of conics to get 
inequivalent families of matrices of constant rank $e=\frac{n(n+3)}{2}-2$ or $3$. 

For a generic pencil $\langle t_1, t_2\rangle $ in $S^2V$ the two tensors $t_1, t_2$ are simultaneously diagonalizable: there exists a basis $e_0,\ldots , e_n$ and coefficients 
$\lambda_0,\ldots ,\mu_n$ such that 
$$t_1=\lambda_0e_0^2+\cdots +\lambda_ne_n^2, \qquad t_2=\mu_0e_0^2+\cdots +\mu_ne_n^2.$$
Some explicit computations yield the following conclusion:

\begin{prop}\label{corank2}
 For $P=\langle t_1, t_2\rangle^\perp\subset S^2V^\vee$, the vector bundle $\cE_P$ is an ordinary Dr\'ezet bundle and $\cE_P^\vee(1)$ is generated by global sections if and only if 
 the pencil $\langle t_1, t_2\rangle$ contains no tensor of rank two.
\end{prop}

\begin{example}
Suppose $n=3$ and consider the pencil generated by 
$$t_1=e_0^2+Te_2^2-e_3^2, \qquad t_2=e_1^2-e_2^2+e_3^2.$$
The other two rank three tensors in the pencil are $t_1+t_2$ and $t_1+Tt_2$, which with $t_1$
and $t_2$ gives four points on the pencil with cross-ratio $T$. In particular, we will 
get inequivalent matrices of linear forms by varying $T$. 

In order to get these matrices, as always we have to choose basis of 
$P=\langle t_1, t_2\rangle^\perp$ and $Syz_1(P)$, whose dimensions are $8$ and $12$, respectively. For a basis of $P$ we choose the six quadrics $Q_{ij}=x_ix_j$, plus 
$Q_0=x_1^2+Tx_3^2-x_4^2$ and $Q_1=x_2^2-x_3^2+x_4^2$. The syzygies are:
$$x_1Q_{23}=x_2Q_{13}=x_3Q_{12}, \qquad x_1Q_{24}=x_2Q_{14}=x_4Q_{12},$$
$$x_1Q_{34}=x_3Q_{14}=x_4Q_{13}, \qquad x_2Q_{34}=x_3Q_{24}=x_4Q_{23},$$
$$x_1Q_1-x_2Q_{12}+x_3Q_{13}-x_4Q_{14}=0,$$
$$x_2Q_0-x_1Q_{12}-Tx_3Q_{23}+x_4Q_{24}=0,$$
$$x_3(Q_0+TQ_1)-x_1Q_{13}-Tx_2Q_{23}+(1-T)x_4Q_{34}=0,$$
$$x_4(Q_0+Q_1)-x_1Q_{14}-x_2Q_{24}+(1-T)x_3Q_{34}=0.$$
We interprete each of the twelve $=$ signs in these formulas as one of our twelve
basis of syzygies. Then we immediately read the associated $8\times 12$ matrix of linear forms
of constant rank $7$, depnding on the parameter $T\ne 0,1$: 

$$\begin{pmatrix}
    x_3&-x_2&0&0&0&0&0&0 \\
    0&x_2&-x_1&0&0&0&0&0 \\
    x_4&0&-x_2&0&0&0&0&0 \\
    0&0&x_2&0&-x_1&0&0&0 \\
    0&x_4&-x_3&0&0&0&0&0 \\
    0&0&x_3&0&0&-x_1&0&0 \\
    0&0&0&x_4&-x_3&0&0&0 \\
    0&0&0&0&x_3&-x_2&0&0 \\
    -x_2&x_3&-x_4&0&0&0&0&x_1 \\
    -x_1&0&0&-Tx_3&x_4&0&x_2&0 \\
    0&-x_1&0&-Tx_2&0&(1-T)x_4&x_3&Tx_3 \\
    0&0&-x_1&0&-x_2&(1-T)x_3&x_4&x_4 
\end{pmatrix}$$

\end{example}

\medskip

\subsection{Matrices of constant rank four}
We have now enough information
to classify matrices of linear forms of constant rank four. 
(The question of classifying matrices of linear forms of rank everywhere 
bounded by four is raised in \cite{MR954659}).
Let $\cE$ be the rank four vector bundle on $\PP^n$ defined by such a matrix. We may suppose 
that $\cE$ is indecomposable. Moreover since $c_1(\cE)+c_1(\cE^\vee(1))=4$, up to exchanging
the two vector bundles, that is, up to transposing the matrix, we may suppose that $c_1(\cE)\le 2$. 
If $c_1(\cE)=1$, then $\cE$ must be the quotient bundle $Q$ on $\PP^4$. 

So suppose $c_1(\cE)=2$. If the rank $e=4\le n+1$, we know that $\cE$ is a direct sum of 
homogeneous bundles, in particular it must be decomposable. So we must have $n=2$.  
According to Propositions \ref{classification-c=2} and \ref{steiner-c=2}, $\cE$ must then be a Drézet bundle $\cE_P$, 
for $P$ a hyperplane in $S^2V^\vee$ (where $\PP^2=\PP(V)$). By Proposition \ref{corank1}, the orthogonal of $P$ in $S^2V$ must be spanned by a non-degenerate tensor. We deduce:

\begin{prop}\label{rankfour}
Any indecomposable matrix of linear forms of constant rank four can be obtained from one of 
the following two matrices:
$$M_{Quot}=\begin{pmatrix}
 x_2&x_3&x_4&x_5&0&0&0&0&0&0\\    
 -x_1&0&0&0&x_3&x_4&x_5&0&0&0\\ 
 0&-x_1&0&0&-x_2&0&0&x_4&x_5&0\\ 
 0&0&-x_1&0&0&-x_2&0&-x_3&0&x_5\\ 
 0&0&0&-x_1&0&0&-x_2&0&-x_3&-x_4\\ 
\end{pmatrix},$$

$$
M_{Drez}=\begin{pmatrix}
 0&x_1&0&0&x_3\\    
 -x_1&0&x_2&x_3&0\\    
 0&-x_2&0&x_1&0\\    
 0& -x_3&-x_1&0&x_2\\    
 -x_3&0&0&-x_2&0\\    
\end{pmatrix}.$$
\end{prop}

\medskip\noindent {\it Remarks}. The first case corresponds to $\cE=Q$, the quotient bundle on $\PP^4=\PP(W)$, for which $H^0(\cE)=W$ and $H^0(\cE^\vee(1))=\wedge^2W^\vee$. Obviously $\cE$ cannot be generated by a proper subspace of global sections, but $\cE^\vee(1)$ can: by general principles $4+4=8$ general sections always suffice. Rather unexpectedly, if $h$ is the hyperplane class, we can compute that  
$$s(\cE^\vee(1))=\frac{1-2h}{(1-h)^5}=1+3h+5h^2+5h^3+{\bf 0}h^4.$$
This means that $7$ general sections (but no less) will suffice to generate $\cE^\vee(1)$
at every point. Concretely, this means that we obtain a matrix of constant rank four by
picking $7$ (or more) general linear combinations of the columns of $M_{Quot}$. And we get in this way all such matrices with associated vector  bundle $Q$. Note also the following consequence of the previous discussion: if we restrict
such a matrix to a smaller space of variables, then it necessarily becomes decomposable. 

The second case corresponds to $\cE=\cE_P$, the Dr\'ezet bundle on $\PP^2=\PP(V)$ 
defined by the orthogonal $P$ to an non-degenerate tensor in $S^2V$. In this case 
$H^0(\cE)=P$ and $H^0(\cE^\vee(1))=Syz_1(P)$ are both five-dimensional, so we need 
their whole spaces of global sections to generate $\cE$ and $\cE^\vee(1))$. We represented this case by an explicit skew-symmetric matrix in three variables, but we know that any generic
such matrix would be equivalent to this one. 

 \proof We have seen that the rank four vector  bundle $\cE$ associated 
to our matrix must be either the quotient bundle $Q$ on $\PP^4$, leading to 
$M_{Quot}$, or the above Drézet bundle $\cE_P$ on $\PP^2$. So the only thing 
we have to prove  is that the skew-symmetric matrix $M_{Drez}$ has constant 
rank four and defines the same Drézet bundle. 

Seen as a space of skew-symmetric forms in $5$ variables, $M_{Drez}$ is generated by $\omega_1=e_1\wedge e_2+e_3\wedge e_4$,
$\omega_2=e_2\wedge e_3+e_4\wedge e_5$, $\omega_3=e_1\wedge e_5+e_2\wedge e_4$. 
It is easy to check that no nonzero linear combination $\omega$ of $\omega_1, 
\omega_2, \omega_3$ satisfies the equation $\omega\wedge\omega=0$, which would mean 
it has rank two. Moreover, recall that the map $\omega\mapsto \omega\wedge\omega$ is precisely
the defining morphism of the associated Steiner bundle $\cE$. A straightforward computation 
shows that the unique relation between the wedge products of the $\omega_i$'s is 
$$\omega_1\wedge \omega_2-\frac{1}{2}\omega_3\wedge \omega_3=0.$$
Since the tensor $t=u_1u_2-\frac{1}{2}u_3^2$ has rank three, we conclude that $\cE=\cE_P$,
which completes the proof. \qed

\subsection{Matrices of constant rank five}
The previous analysis can be extended to matrices of linear forms 
of constant rank five.
Let $\cE$ be the rank five vector bundle on $\PP^n$ defined by such a matrix, and suppose 
that $\cE$ is indecomposable. Now $c_1(\cE)+c_1(\cE^\vee(1))=5$, so again, 
up to exchanging the two vector bundles, we may suppose that $c_1(\cE)\le 2$. 
If $c_1(\cE)=1$, then $\cE$ must be the quotient bundle $Q$ on $\PP^5$. 

So suppose $c_1(\cE)=2$. The condition $e=5\ge n+2$, implies $n\le 3$. Moreover,
as in the previous case, 
Propositions \ref{classification-c=2} and \ref{steiner-c=2} imply that 
$\cE$ is a Drézet bundle $\cE_P$ on $\PP^n=\PP(V)$
for $P\subset S^2V^\vee$ of dimension $e+1=6$.

If $n=2$, this implies that $P=S^2V^\vee$ and  then $\cE$ is the universal Drézet 
bundle discussed in Example \ref{drezet-steiner}.

Let us show that $n=3$ is impossible for an ordinary Drézet bundle. 
First note that $s(\cE^\vee(1))=(1+3h)/(1+h)^6$ and therefore $s_3(\cE^\vee(1))\ne 0$. This implies that 
$\cE^\vee(1)$, if generated by global sections, has at least $e+n=8$ independant sections. Since 
$\chi(\cE^\vee(1))=4$ and  $h^q(\cE^\vee(1))=0$ for $q>1$, this means that 
$h^1(\cE^\vee(1))\ge 4$. Therefore, if  $\cE=\cE_P$ there must exist a
four-dimensional vector space of cubics, all apolar to $P$.

\begin{example} 
The six-dimensional space of quadrics generated by 
the square-free monomials $P=\langle q_{ij}=x_ix_j, i< j\rangle$ is apolar to the four dimensional
space of Fermat type cubic tensors $C=c_0e_0^3+c_1e_1^3+c_2e_2^3+c_3e_3^3$. In this case
nevertheless, $P$ has four base points. As a consequence we do not get a constant rank matrix 
of linear forms, but we do get a $6\times 8$ matrix of linear forms of generic rank $5$. 
The eight obvious syzygies are $x_iq_{jk}=x_jq_{ik}=x_kq_{ij}$ for $i<j<k$, and we get the 
matrix 
$$\begin{pmatrix}
    0&-x_2&0&x_1&0&0 \\
    x_3&0&0&-x_1&0&0 \\
    x_4&0&-x_2&0&0&0 \\
    -x_4&0&0&0&x_1&0 \\
    0&-x_4&x_3&0&0&0 \\
    0&0&-x_3&0&0&x_1 \\
    0&0&0&0&-x_3&x_2 \\
    0&0&0&x_4&0&-x_2 
\end{pmatrix}$$
Tautologically, 
the kernel is spanned outside the four base points by $\sum_{i<j}x_ix_jq_{ij}$. 
\end{example}

\medskip Let us show this cannot happen when we do not admit base points. 

In order to prove our next statements,  we will need to recall a couple of pure algebraic results which will play an important role in the proofs of Lemmas \ref{mac} and \ref{mac2}.

\vskip 2mm
Given a standard graded Artinian $K$-algebra $A=R/I$ where $R=K[x_0,x_1,\dots,x_n]$ and $I$ is a homogeneous ideal of $R$,
we denote its Hilbert function by $HF_A:\mathbb{N} \longrightarrow \mathbb{N}$, where $HF_A(j)=dim _KA_j=dim _K[R/I]_j$. Since $A$ is Artinian, its Hilbert function is
captured in its \emph{$h$-vector} $h=(h_0,h_1,\dots ,h_d)$ where $h_i=HF_A(i)>0$ and $d$ is the last index with this property. The integer $d$ is called the \emph{socle degree of} $A$.

Given  integers $n, r\ge 1$, define the  \emph{$r$-th binomial expansion of $n$} as
\[
n=\binom{m_r}{r}+\binom{m_{r-1}}{r-1}+\cdots +\binom{m_e}{e}
\]
where $m_r>m_{r-1}>\cdots >m_e\ge e\ge 1$ 
are uniquely determined integers (see \cite[Lemma 4.2.6]{MR1251956}).
Write
\[ 
n^{<r>}=\binom{m_r+1}{r+1}+\binom{m_{r-1}+1}{r}+\cdots +\binom{m_e+1}{e+1},
\]
\[
n_{<r>}=\binom{m_r-1}{r}+\binom{m_{r-1}-1}{r-1}+\cdots +\binom{m_e-1}{e}.
\]

The numerical functions $H:\mathbb N \longrightarrow \mathbb N$ that are Hilbert functions of standard graded $K$-algebras were characterized by Macaulay  \cite{MR1251956}. Given a numerical function $H:\mathbb N \longrightarrow \mathbb N$, the following conditions are equivalent:
\begin{enumerate}
\item[(i)] there exists a standard graded $K$-algebra $A$ with $H$ as Hilbert function,
\item[(ii)] $H(0)=1$ and $H$ satisfies the so-called \textbf{Macaulay's inequalities}:
\end{enumerate}
$$
H(t+1)\le H(t)^{<t>}\quad\forall t\ge 1.
$$

Notice that condition (ii) imposes strong restrictions on the Hilbert function of a standard graded $K$-algebra, in particular it bounds its growth.

\begin{example} Let $A$ be a standard graded Artinian $K$-algebra with $h$-vector $h=(h_0,h_1,\dots ,h_d)$. If $h_3\le 3$ then $h_t\le 3$ for all $t\ge 3$.
\end{example}

\begin{prop}\label{link} Let $R/I_1$ and $R/I_2$ be two standard graded Artinian $K$-algebras with $h$-vectors $(h^1_0,h^1_1,\dots ,h^1_d)$ and  $(h^2_0,h^2_1,\dots ,h^2_c)$, respectively. Assume that $I_1$ and $I_2$ are linked by  an Artinian complete intersection $K$-algebra $R/J$ with h-vector $(1,n+1,h_2,\dots ,h_{r-2},n+1,1)$. Then
$$
h_i-h^1_i=h^2_{r-i} \  \text{ for all } i.
$$
\end{prop}
\proof
Let us denote by $K_2$ the canonical module of $R/I_2$. Since $I_1$ and $I_2$ are linked via $J$, by definition of linkage we have $I_2=[J:I_1]$. Moreover there is
an exact sequence \cite[Proposition 2.1.1]{MR2375719}:
$$
0 \to J \to I_1 \to K_2(-r) \to 0.
$$
Therefore we get, by definition of the canonical module,  
$$\begin{array}{rcl}
h_i-h^1_i & = & dim [K_2(-r)]_i \\
& = & dim [K_2]_{-r+i} \\
& = & dim[Hom(R/I_2,K)]_{-r+i} \\
& = & dim [R/I_2]_{r-i} \\
& = & h^2_{r-i}
\end{array}
$$ which proves the claim. \qed

\begin{lemma}\label{mac} For any six-dimensional subspace $P\subset S^2V^\vee$, consisting of quadrics
with no non-trivial common zero, the rank five 
Drézet bundle on $\PP V=\PP^3$ verifies $h^1(\cE^\vee(1))\le 3$. As a consequence 
$\cE^\vee(1)$  cannot be generated by global sections. 
\end{lemma}

\proof From the exact cohomology sequence associated to the exact sequence
$$0 \longrightarrow \cE^\vee(1) \longrightarrow P \otimes \cO _{\PP^3} (1)\longrightarrow \cO _{\PP^3} (3)\longrightarrow 0 $$
we immediately deduce that $H^1(\cE^\vee(1))\cong (R/I)_3$ where $I\subset K[x,y,z,t]$ is the ideal generated by our six quadrics. Therefore, we only need to check that  
$ dim (R/I)_3\le 3$.   Inside $I$ we choose 4 quadrics $Q_1,\dots ,Q_4$ defining a complete intersection ideal $J\subset R$. Let us call $h_t$ the Hilbert function of $R/J$
and $h'_t$ the Hilbert function of $R/I$. So, we have $(h_t)=(1,4,6,4,1,0,....)$ and  $(h'_t)=(1,4,4,h'_3,h'_4,0,........)$. Let us call $I_1$ the ideal directly linked to $I$ by means of the complete intersection $J$ and $(h''_t)=(1,h''_1,h''_2,h''_3,h''_4,0,........)$ its Hilbert function. The Hilbert function of $R/I_1$ is determined by the Hilbert functions of $R/I$ and $R/J$ \cite[Theorem 3]{MR776185} (see also Proposition \ref{link}),
and must verify Macaulay's inequalities (see \cite{MR1251956}).  In particular,  
the fact that $h''_2=h_2-h'_2=6-4=2$ implies  that $h''_1=h_3-h'_3\ge 2$. Thus $h'_3< 3$, as claimed. \qed

\medskip The conclusion of this discussion is the following statement:

\begin{prop}
    Any indecomposable matrix of linear forms of constant rank five comes  from an 
    equivariant vector bundle. More precisely, the associated bundle $\cE$ must be either:
    \begin{enumerate}
        \item the quotient bundle $Q$ on $\PP^5$, or
        \item the homogeneous universal Drézet bundle on $\PP^2$, which fits into an extension $0\ra Q(-1)\ra \cE\ra S^2Q\ra 0$. 
    \end{enumerate}\end{prop}

In the first case, $H^0(\cE)=V$ and $H^0(\cE^\vee (1))=\wedge^2V^\vee$ have dimension $6$ and 
$15$, hence a $6\times 15$ matrix of constant rank five given by the natural contraction map 
$V\otimes \wedge^2V^\vee\ra V^\vee$. Any constant rank matrix defining the same vector bundle
is then obtained by choosing at least $9$ general linear combinations of the $15$ columns. 

In the second case, that we first met in Examples \ref{steiner-drezet} and \ref{drezet-steiner}, 
we have $H^0(\cE)=S^2V$ and 
$H^0(\cE^\vee (1))=S_{21}V^\vee\subset S^2V^\vee\otimes V^\vee$, of dimensions $6$ and $8$.
The corresponding matrix of constant rank five is given by the natural contraction map 
$S^2V\otimes S_{21}V^\vee\ra V^\vee$ and appears in Example \ref{steiner-drezet}.
Any constant rank matrix defining the same vector bundle
is then obtained by choosing at least $7$ general linear combinations of the $8$ columns.

\section{Classification for $c_1=3$ and matrices of constant rank six}

\subsection{Bundles with $c_1=3$} 
Globally generated vector bundles $\cE$ with $c_1(\cE)= 3$ on $\PP^n$
have been classified. Our goal is to extract from this classification the list
of indecomposable vector bundles such that $\cE^\vee(1)$ is also generated. The indecomposability 
condition imposes that $e\ge n+2$, and also $h^0(\cE^\vee)=h^0(\cE(-1))=0$. Moreover
we may suppose that $e\ge 6$, since otherwise $c_1(\cE^\vee(1))<3$ and therefore 
 $\cE^\vee(1)$ must have been classified previously.

Under these hypothesis,  \cite[Theorem (v)]{anghel-manolache} implies that if $n\ge 5$,
$\cE$  must be either a Steiner or a Drezet bundle, or given by 
an exact sequence of type 
$$0\ra \cO_{\PP^n}(-1)\oplus \cO_{\PP^n}(-2) \ra  \cO_{\PP^n}^{\oplus e+2}\ra \cE\ra 0.$$

On $\PP^4=\PP V$, a specific indecomposable and globally generated bundle is  $\cE=\wedge^2Q=\cE^\vee(1)$. In this case $H^0(\cE)=H^0(\cE^\vee(1))=\wedge^2V$ 
is ten-dimensional, and the associated $10\times 10$ matrix of constant rank $6$ 
is the matrix of the symmetric bilinear form $\wedge^2V\otimes\wedge^2V\ra\wedge^4V
\simeq V^\vee$. In the basis $e_i\wedge e_j$, $i<j$ ordered lexicographically we get the symmetric matrix 
$$\begin{pmatrix}
0&0&0&0&0&0&0&x_5&-x_4&x_3 \\
0&0&0&0&0&-x_5&x_4&0&0&-x_2 \\
0&0&0&0&x_5&0&-x_3&0&x_2&0 \\
0&0&0&0&-x_4&x_3&0&-x_2&0&0 \\
0&0&x_5&-x_4&0&0&0&0&0&x_1 \\
0&-x_5&0&x_3&0&0&0&0&-x_1&0 \\
0&x_4&-x_3&0&0&0&0&x_1&0&0 \\
x_5&0&0&-x_2&0&0&x_1&0&0&0 \\
-x_4&0&x_2&0&0&-x_1&0&0&0&0 \\
x_3&-x_2&0&0&x_1&0&0&0&0&0 
\end{pmatrix}$$
Note that $s(\wedge^2Q)=(1-h)^5/(1-2h)$, hence $s_4(\wedge^2Q)=1\ne 0$. As a consequence,
this matrix cannot be reduced to a smaller one without droping rank somewhere. 

Otherwise, by \cite[Theorem (iv)]{anghel-manolache}, $\cE$ must be an extension by a trivial vector bundle, of a rank four vector  bundle $\cF$ which must belong to one of the types 
listed in  \cite[Theorem (ii)]{anghel-manolache}. Since then $h^0(\cE(-1))=h^0(\cF(-1))$
and $h^0(\cE(-2))=h^0(\cF(-2))$, types (1-4) can be eliminated. Moreover, in type (7) 
the vector bundle $\cF=Q^\vee(1)$ has no non-trivial extension since the Bott-Borel-Weil theorem
implies that $H^1(\cF^\vee)=0$; so this type can be eliminated as well. We thus remain
with the same three types as on $\PP^n$ for $n\ge 5$.

On $\PP^3$, a full classification was given in \cite{manolache}, and as in the previous 
case, a globally generated bundle $\cE$ of rank $e>3$ must be an extension by a trivial bundle, of a rank vector three bundle $\cF$  of one of the types 
listed in  \cite[Theorem (ii)]{manolache}. For the same reasons as before we can eliminate 
cases (2-3-4-6-8) in this list, and we remain with the same three types as before (Drézet, Steiner and mixed) plus case (9), which we call the Tango case. 

Finally on $\PP^2$, the only additional possibility is to consider an extension of the 
tangent bundle $T= Q(1)$ by a trivial bundle. But since $h^1(\Omega^1)=1$, only an 
extension by a trivial line bundle can yield an indecomposable vector bundle $\cE$, in fact still
homogeneous. But a quick computation in this case shows that $h^0(\cE^\vee(1))=3$, so that
$\cE^\vee(1)$ cannot be generated by global sections. 

We have proved:

\begin{prop}
Suppose that $\cE$ is an indecomposable vector bundle of rank $e\ge 6$ on $\PP^n$, 
with $c_1(\cE)=3$. Suppose that $\cE$ and $\cE^\vee(1)$ are both globally generated.  
Then either:
\begin{enumerate}
    \item $\cE$ is a Steiner bundle, or
    \item $\cE$ is a Drézet bundle, or
    \item $\cE$ fits into a sequence $0\ra \cO_{\PP^n}(-1)\oplus \cO_{\PP^n}(-2)\ra \cO_{\PP^n}^{\oplus e+2}\ra\cE\ra 0$, 
    \item  or $n=4$ and $\cE=\wedge^2Q$.
\end{enumerate}
\end{prop}

We have already discussed the last case. Let us consider the other 
types of vector bundles of rank six. 

\subsubsection{The Steiner case} 
A rank six Steiner bundle $\cE$ on $\PP^n=\PP V$, for $n\le 4$, is  
given by an exact sequence 
$$0\ra A_3\otimes \cO_{\PP^n}(-1)\ra B_9\otimes \cO_{\PP^n}\ra\cE\ra 0.$$
If $\cE$ is uniform, by Proposition \ref{miroroigmarchesi} we need $3+2n-2\le 6$, hence $n=2$. Then $\cE$ must 
be the sum of three copies of $Q$. In particular, it is not indecomposable.  

\subsubsection{The Drézet case} 
Consider a rank six Drézet bundle $\cE_P$ on $\PP^n=\PP V$, for $n\le 4$, 
given by an exact sequence 
$$0\ra \cO_{\PP^n}(-3)\ra P^\vee\otimes \cO_{\PP^n}\ra\cE_P\ra 0.$$
Here $P\subset S^3V^\vee$ is a seven-dimensional space of cubics, and we wonder whether $\cE_P^\vee(1)$ can be generated by global sections. 

\begin{lemma}\label{mac2} If $n=3,4$, $\cE_P^\vee(1)$ cannot be generated by global sections.

If $n=2$, then $\cE_P^\vee(1)$ is generated by global sections if and only if it is also a Drézet bundle, and in this case $\cE_P$ is not ordinary. \end{lemma}

\proof The expected dimension of $H^0(\cE_P^\vee(1))$ is $0$ for $n=4$ and $n=3$; and $6$ for $n=2$.
For $n=4$, if $\cE_P$ is an ordinary Drézet bundle we have $h^0(\cE_P^\vee(1))=0$ while for any seven-dimensional subspace $P\subset S^3V^\vee$ consisting of cubics without common zeros we have $ h^1(\cE_P^\vee(1))\le 44$, $h^0(\cE_P^\vee(1))\le 9$ and hence we do not have enough linear syzygies.

For $n=3$, if  $\cE_P$ is an ordinary Drézet bundle we have again $H^0\cE_P^\vee(1)=0$.  Otherwise we argue as in Lemma \ref{mac}.  We denote by $I$ the ideal generated by the seven cubics and by $J$ a complete intersection Artinian ideal generated by four cubics in $I$. We use $J$ to link $I$ to $I_1$. We denote by $h_t$, $h_t'$ and $h_t''$ the Hilbert functions of $J$, $I$ and $I_1$, respectively. By Proposition \ref{link} we know that for $i\ge 0$, $h_{i}''=h_{8-i}-h_{8-i}'$. In particular, we have $h_5''=h_3-h_3'=16-13=3$. Applying Macaulay's inequalities we get $h_4''=h_4-h_4'=19-h_4'\ge 3$ and by Gotzmann persistence theorem \cite{MR480478} we have $h_4''>3$. Therefore, we get that $ h^1(\cE_P^\vee(1))=h_4'=h_4-h_4''\le 15 $ and, hence, $h^0(\cE_P^\vee (1))\le 8 $ and we do not have enough linear syzygies.

\smallskip
For $n=2$, a computation shows that $s_2(\cE_P^\vee(1))=0$. This implies that if
$\cE_P^\vee(1)$ is generated by global sections, it can be generated by $e+n-1=7$
sections, and we get an exact sequence $0\ra L\ra \cO_{\PP^2}^{\oplus 7}
\ra\cE_P^\vee(1)\ra 0$. Adjusting Chern classes yields $L=\cO_{\PP^2}(-3)$. 
As a consequence $h^1(\cE_P^\vee(1))=h^2(L)=1$, which means that $\cE_P$ is 
not ordinary. \qed

\medskip Note that we get an exact complex 
$$0\ra \cO_{\PP^n}(-3)\ra P^\vee\otimes \cO_{\PP^n}\stackrel{M}{\ra} Q\otimes \cO_{\PP^n}(1)\ra
\cO_{\PP^n}(4)\ra 0,$$
with $P^\vee=H^0(\cE_P)$ and $Q^\vee=H^0(\cE_P^\vee(1))$. According to \cite{MR0453723},
one can then identify $P$ and $Q$, and $M$ is then skew-symmetric. In other words,
the bundle $\cE_P$ must be Pfaffian. The corresponding $7\times 7$ matrix of constant 
rank $6$ is then obtained by restricting the universal skew-symmetric $7\times 7$ matrix
to a projective plane that does not meet the codimension three locus of matrices 
of rank at most four. Finally, $P$ is generated by the seven Pfaffian minors of this 
matrix. 

\begin{example}
Consider the smooth plane quartic $C$ whose equation is given by the tensor 
$e_0^3e_1+e_1^3e_2+e_2^3e_0$. The seven dimensional space of cubics apolar 
to this quartic tensor is
$$P=\langle x_0^3-x_1^2x_2, x_1^3-x_2^2x_0, x_2^3-x_0^2x_1, 
x_0^2x_2, x_1^2x_0, x_2^2x_1, x_0x_1x_2\rangle .$$
Denote the seven generators of $P$ by $A_0, A_1, A_2, B_0, B_1, B_2, E$. 
They admit seven linear syzygies, and $Syz_1(P)$ is generated by 
$x_1B_0-x_0E$, $x_2B_1-x_1E$, $x_0B_2-x_2E$, $x_2B_2-x_0B_1-x_1A_2$,  
$x_0B_0-x_1B_2-x_2A_0$, $x_1B_1-x_2B_0-x_0A_1$, $x_0A_0+x_1A_1+x_2A_2$. 
Idexing the columns by $A_2, A_1, A_0, B_0, B_1, B_2, E$ these seven syzygies
yield the skew-symmetric matrix
$$\begin{pmatrix}
 0&0&0&x_1&0&0&-x_0 \\
 0&0&0&0&x_2&0&-x_1 \\
 0&0&0&0&0&x_0&-x_2 \\
 -x_1&0&0&0&-x_0&x_2&0 \\
 0&-x_2&0&x_0&0&-x_1&0 \\
 0&0&-x_0&-x_2&x_1&0&0 \\
 x_0&x_1&x_2&0&0&0&0 
\end{pmatrix}$$
\end{example}

Using a normal form for plane quartics it is in principle possible to obtain
a normal form for all $7\times 7$ matrices of constant rank $6$ of linear forms in three 
variables. Note that the dimension of $\PP(S^4V)/PGL(V)$ is $6$, and that this is 
also the dimension of $G(3,\wedge^2A_7)/PGL(A_7)$. 

We can summarize this discussion as follows.

\begin{prop} All $7\times 7$ matrices of constant rank $6$ of linear forms in three variables are given by Drézet bundles on $\PP^2$ which are not ordinary. These
Drézet bundles are in fact Pfaffian, and defined by linear systems of cubics 
which are apolar to a given quartic curve in the dual projective plane. 
\end{prop}

\subsubsection{The mixed  case}
Consider now a vector bundle $\cE$ on $\PP^n=\PP V$, for $n\le 4$, defined   
by an exact sequence 
$$0\ra \cO_{\PP^n}(-1)\oplus \cO_{\PP^n}(-2)\ra A^\vee\otimes 
\cO_{\PP^n}\ra\cE\ra 0,$$
with $A$ of dimension eight. Supposing that $h^0(\cE^\vee)=0$, the
dual sequence implies that $A$ embeds into $V^\vee\oplus S^2V^\vee$. 

\medskip Suppose first that $n=2$, so that $A$ is a hyperplane in 
$V^\vee\oplus S^2V^\vee$. Denote by $(a_0,q_0)\in 
V\oplus S^2V$ a generator of its orthogonal. The fiber of $\cE$ at $[v]$
can then be identified with $(V\oplus S^2V)/\langle (v,0), (0,v^2), (a_0,q_0)
\rangle$. In particular there is a natural morphism from $Q$ to $\cE$, which 
is an injective morphism of vector bundles when $q_0$ has rank at least two. 
Under this hypothesis, we get an exact sequence 
$$0\ra Q\ra \cE\ra\cE_H\ra 0,$$
where $\cE_H$ is the rank four Drézet bundle defined by $q_0$, so that $H$ 
is the orthogonal hyperplane to $q_0$ in $S^2V^\perp$. Such bundles 
were discussed in Example \ref{drezet-steiner}.

\begin{lemma}
This exact sequence  splits if and only if $a_0$ belongs to the span of $q_0$.
In particular, it  always splits when $q_0$ has rank three, but not necessarily
when $q_0$ has rank two. 
\end{lemma}

\proof A splitting of the exact sequence must be induced by a morphism $\iota : S^2V\ra V$ such that $\iota(v^2)$ is a multiple of $v$ for all $v$, and $\iota(q_0)=a_0$. The first condition implies that
$\iota$ is the contraction map by 
some linear form $\theta$. Then the second condition means that 
$a_0=\theta\rfloor q_0$, and the existence of such a $\theta$ precisely 
means that $a_0$ belongs to the span of $q_0$. \qed

\medskip  As a consequence, $\cE$ is decomposable unless $q_0$ has rank two and $a_0$ does not belong to its span. In this case we know by the second part of Example \ref{drezet-steiner} that $\cF^\vee(1)$ is not generated by global sections, but this does not  a priori prevent $\cE^\vee(1)$ to be generated. 
And this is exactly what happens!

\begin{prop}
Suppose $q_0$ has rank two, and $a_0$ does not belong to its span. 
Then $\cE^\vee(1)$ is generated by global sections.
\end{prop}

\proof 
As above we can choose a basis such that $(a_0,q_0)=(e_0, e_1e_2)$. 
We first claim that  $h^0(\cE^\vee(1))=8$, the expected value. 
Indeed, the exact sequence 
$$0\ra H^0(\cE^\vee(1))\ra A\otimes V^\vee\ra S^2V^\vee\oplus S^3V^\vee
\ra H^1(\cE^\vee(1))\ra 0$$
allows to interprete $H^0(\cE^\vee(1))$ as a space of linear syzygies
between vectors in $A$. Choose a basis of $A$ given by the vectors 
$(-x_0,m)$,  $(x_1,0), (x_2,0)$, and $(0,n)$ for the five monomials $n$ 
distinct from $m=x_1x_2$ or $x_2^2$. Then the linear syzygies of $A$ 
are generated by the tautological linear syzygy $x_2(x_1,0)-x_1(x_2,0)$, and 
the  syzygies of  the form
$$\sum_{n\ne m} \ell_n (0,n)+\ell_m (-x_0 , m) +\ell_1(x_1,0)+\ell_2(x_2,0)=0$$
where $\sum_{n} \ell_n n$ is a non-trivial syzygy between the six degree two monomials in $x_0, x_1, x_2$; this imposes that $\ell_m$ has no term in $x_0$,
so there are seven such syzygies, which makes eight with the tautological 
one. 

As a consequence, $h^1(\cE^\vee(1))=0$, and we are allowed to use the snake 
Lemma exactly as we did for Drézet bundles in the proof of Lemma \ref{criterion}.
We conclude that $\cE^\vee(1)$ is generated at $[v]$ if and only if the natural
morphism 
$$S^2V/\langle v^2\rangle \oplus S^3V/\langle v^3\rangle
\ra Hom(A,V/\langle v\rangle)$$
is injective. If this is not the case, there exists tensors $Q\in S^2V$ and 
$C\in S^3V$, not both powers of $v$, such that for any $(\ell, \kappa)$ in 
$A$ the contraction 
 $\ell\rfloor Q+\kappa\rfloor C$ gives a vector in $\langle v\rangle$.
Let us make an explicit computation. Using our prefered basis of $A$ 
we rewrite these conditions as
$$x_0\rfloor Q=x_1x_2\rfloor C, \quad x_1\rfloor Q= x_2\rfloor Q= n\rfloor C=0
\quad mod \; v$$
for $n$ any monomial other that $x_1x_2$. In particular there must exist scalars 
$s_i$ such that $x_i^2\rfloor C=s_i v$, and we deduce that 
$$C=\frac{1}{6}\sum_i s_iv_ie_i^3+\frac{1}{2}\sum_{i\ne j} s_iv_je_i^2e_i+
te_0e_1e_2$$
for some $t$. There must also exist $r_1,r_2$ such that $x_0x_i\rfloor C=r_i v$,
which gives
$$t=r_1v_2, \quad s_0v_1=r_1v_0, \quad s_1v_0=r_1v_1,$$
$$t=r_2v_1, \quad s_0v_2=r_2v_0, \quad s_2v_0=r_2v_2.$$
This implies that $s_0v_1^2=s_1v_0^2$ and $s_0v_2^2=s_2v_0^2$. 

If $(s_0,s_1,s_2)$ is colinear to $(v_0^2, v_1^2, v_2^2)$, then $C$ reduces to $t'e_0e_1e_2$ modulo $v^3$, for some scalar $t'$. Then we need $t'e_1$ and $t'e_2$ to be both multiples of $v$, which implies that $t'=0$. Finally, we remain with the conditions that $x_i\rfloor Q$ be multiples of $v$ for each $i$, which is only possible if $Q$ is a multiple of $v^2$. 

The other possibility is that $s_0=v_0=0$, which implies that $r_1v_1=r_2v_2=0$. 
If $v_1v_2\ne 0$, we get $r_1=r_2=0$ and we deduce that $C$ has no term involving $e_0$, and that we can write
$$C=\frac{1}{6}(s_1v_1e_1^3+s_2v_2e_2^3)+
\frac{1}{2}(s_1v_2e_1^2e_2+s_2v_1e_1e_2^2).$$
We deduce that $x_0\rfloor Q=x_1x_2\rfloor C=s_1v_2e_1+s_2v_1e_2$ modulo $v$. 
Since the three contractions of $Q$ by $x_0, x_1, x_2$ do not involve
$e_0$, this must also be the case of $Q$, so that in fact $x_0\rfloor Q=0$. 
So $s_1v_2e_1+s_2v_1e_2$ must be a multiple of $v=v_1e_1+v_2e_2$, which amounts 
to $s_1v_2^2=s_2v_1^2$ and implies that $C$ is a multiple of $v^3$. Finally, 
the three contractions of $Q$ by $x_0, x_1, x_2$ must then be multiples of $v$,
and thus $Q$ is a multiple of $v^2$. This concludes the proof. \qed 

\medskip 
We immediately get the associated $8\times 8$ matrix of constant rank $6$, 
using the same basis as above for $A$ and its linear syzygies:

$$M=\begin{pmatrix}
0&x_2&-x_1&0&0&0&0&0 \\
-x_2&0&x_0&0&0&-x_1&0&0 \\
x_1&-x_0&0&0&x_2&0&0&0 \\
0&0&0&x_1&0&0&-x_0&0 \\
0&0&0&x_2&0&0&0&-x_0 \\
0&0&0&0&x_0&0&-x_1&0 \\
0&0&0&0&0&x_0&0&-x_2 \\
0&0&0&0&0&0&x_2&-x_1 
\end{pmatrix}.$$

This matrix clearly exhibits the fact that $\cE$ is an extension of the quotient 
bundle, corresponding to the $3\times 3$ NorthWest block, by a non-generic Drézet bundle corresponding to the $5\times 5$ SouthEast block. The former has constant
rank two. The latter has bounded rank four, and we know that the rank does drop
to three on the line $(x_0=0)$; but remarkably, this is cured by the two 
extra nonzero entries of the matrix, which keep the rank constant. 

\smallskip
Note also that there is  an exact complex
  $$0\ra \cO_{\PP^2}(-1)\oplus \cO_{\PP^2}(-2)\ra \cO_{\PP^2}^{\oplus 8}\stackrel{M}{\ra} \cO_{\PP^2}(1)^{\oplus 8}\ra \cO_{\PP^2}(2)\oplus \cO_{\PP^2}(3)\ra  0.$$
 
\begin{prop}
Suppose $q_0$ has rank one, and that $a_0$ does not belong to its span. 
Then $\cE^\vee(1)$ is not generated by global sections.
\end{prop}

\proof An explicit computation, choosing $(a_0,q_0)=(e_0,e_2^2)$, shows that 
again $\cE^\vee(1)$ has eight sections. This allows to use the same 
criterion for global generation as in the previous Lemma. And it is 
easy to exhibit a counter-example: choose $v=e_2$, $C=e_2^2e_0$ and $Q=\frac{1}{2}e_0^2$. 
Indeed, if $n$ is a degree two monomial distinct from $x_2^2$ it is clear 
that $n\rfloor c$ ia a multiple of $e_2$, and it is even clearer that 
$x_1\rfloor Q=x_2\rfloor Q=0$. Finally, $(x_0,-\frac{1}{2}x_2^2)$ is in $A$, 
and applying it to $(Q,C)$ we get $x_0\rfloor Q-\frac{1}{2}x_2^2\rfloor C=0$. \qed

 \medskip
 Now that we know precisely what can happen on $\PP^2$, it is easy to upgrade 
 to $\PP^3$ and $\PP^4$. So again suppose that the vector  bundle $\cE$ is defined 
by an exact sequence 
$$0\ra \cO_{\PP^3}(-1)\oplus \cO_{\PP^3}(-2)\ra A^\vee\otimes 
\cO_{\PP^3}\ra\cE\ra 0,$$
with $A$ of dimension eight. Consider a hyperplane $H\subset \PP^3$ and denote 
by $\cF$ the restriction of $\cE$ to $H$; it is also defined by 
the previous sequence, restricted to $H$. Moreover if $\cE^\vee(1)$ is generated 
by global sections, then $\cF^\vee(1)$ has the same property, so by the previous discussions we know that either $\cF$ splits, or is a non-trivial extension of 
a Drézet bundle by the quotient bundle. Moreover we have seen that $h^0(\cF^\vee(1))=8$ in both cases. But the exact sequence $0\ra \cE^\vee\ra 
\cE^\vee(1)\ra \cF^\vee(1)\ra 0$ implies that if $\cE$ is indecomposable, 
$h^0(\cE^\vee(1))\le h^0( \cF^\vee(1))=8.$ Since $8<6+3$ and $s_3(\cE^\vee(1))=12
\ne 0$, we do not have enough sections for $\cE^\vee(1)$ to be globally 
generated.  A fortiori we cannot have such bundles on $\PP^4$, and we have proved:

 \begin{prop}
 On $\PP^3$ and $\PP^4$, there is no rank six indecomposable bundle $\cE$ of 
 mixed type such that $\cE^\vee(1)$ is globally generated. 
 \end{prop}

\subsubsection{The Tango  case} Finally, there is a family of vector bundles on $\PP^3$ 
connected with the Tango bundle, which can be defined by an exact sequence of type
$$0\ra Q(-1)\oplus \cO_{\PP^3}(-1)\ra A_{e+4}\otimes \cO_{\PP^3}\ra\cE\ra 0.$$
(Remember that the Tango bundle $\cF^\vee$ fits into a sequence $0\ra Q(-1)\ra B_5
\otimes \cO_{\PP^3}\ra\cF^\vee\ra 0.$) Since $h^0(Q^\vee(1))=6$ and $h^0(\cO_{\PP^3}(1))=4$, we must have $h^0(\cE^\vee)>0$ as soon as $e>6$, and then $\cE$ has a trivial factor. 
Moreover if $e=6$, necessarily $A_{10}^\vee\simeq H^0(Q^\vee(1))\oplus H^0(\cO_{\PP^3}(1))$
and then $\cE$ has to decompose as $\wedge^2Q\oplus Q$. As a consequence there is no indecomposable  vector bundle of rank at least six in this family of vector bundles.

\medskip\noindent {\it Remark.} Some extra work would be required to get a full classification for $c_1=3$. This classification would necessarily include Example \ref{3Q-1}.

\subsection{Matrices of constant rank six}

We have just classified those matrices of constant rank $6$ whose associated vector bundle 
$\cE$ has $c_1(\cE)=3$. There remains to understand the cases where $c_1(\cE)<3$, and 
$\cE$ indecomposable.

For $c_1(\cE)=1$, we know that $\cE$ must be the quotient bundle $Q$ on $\PP^6=\PP V$,
and that the possible matrices are submatrices of the $7\times 21$ matrix 
encoding, in some chosen basis, the contraction map $V\otimes \wedge^2V^\vee\ra V^\vee$.

For $c_1(\cE)=2$, we know from Propositions \ref{classification-c=2} and 
\ref{steiner-c=2}
that $\cE$ must be a Drézet bundle on $\PP^n$, with $2\le n\le 4$. 
Since $S^2\CC^3$ has only dimension six, there is no such 
indecomposable rank six Drézet bundle on $\PP^2$. The case of 
$\PP^4$ is also easy to exclude:

\begin{lemma}
On $\PP^4$, $\cE_P^\vee(1)$ cannot be generated by global sections.
\end{lemma}

\proof The expected dimension of $H^0(\cE_P^\vee(1))$ is zero, while we would
need at least ten sections to generate it everywhere. If $\cE_P$ is an ordinary Drézet bundle we have $H^0\cE_P^\vee(1)=0$ while for any seven-dimensional subspace $P\subset S^2V^\vee$ consisting of quadrics without common zeros we have $H^1\cE_P^\vee(1)\le 9$ and hence $H^0\cE_P^\vee(1)\le 9$.  \qed 

\medskip So we can focus on $\PP^3=\PP V$, and consider an exact sequence 
$$0\ra \cO_{\PP^3}(-2)\ra P^\vee\otimes \cO_{\PP^3}\ra\cE_P\ra 0,$$
for some dimension seven subspace $P$ of $S^2V^\vee$, with no non-trivial common 
zero. An easy computation yields:

\begin{lemma} 
For $k\ge -1$, $\cE_P(k)$ has no higher cohomology. Moreover
$$h^0(\cE_P(k))=(k+1)(k^2+6k+7).$$
\end{lemma}
 
Dualizing the defining sequence of $\cE_P$, we get:

\begin{lemma} \label{h1} \ 
\begin{enumerate}
\item
For $k\ge 0$, $h^0(\cE_P^\vee(-k))=0$.
\item For $k\ge 3$, $ h^1(\cE_P^\vee(-k))=0.$ 
\item
For $0\le  k\le 2$ and $q\ne 1$, $h^q(\cE_P^\vee(-k))=0$.
\item
$h^1(\cE_P^\vee)=3, \quad h^1(\cE_P^\vee(-1))=4, \quad h^1(\cE_P^\vee(-2))=1.$
\end{enumerate}
\end{lemma}

Denote by $\fsl(\cE_P)$ the vector bundle of traceless endomorphisms of $\cE_P$.

\begin{prop}
$\cE_P$ is simple, and $h^q(\fsl(\cE_P))=\delta_{q,1}21$. 
\end{prop}

\proof Using the previous lemmas we get  
the long exact sequence 
$$
0\rightarrow H^0(End(\cE_P))\rightarrow H^1(\cE_P^\vee(-2))
\rightarrow P^\vee \otimes H^1(\cE_P^\vee) \rightarrow H^1(End(\cE_P))\rightarrow 0,$$
and $H^q(End(\cE_P))=0$ for $q>1$. Since $H^0(End(\cE_P))$ contains the
homotheties and
$H^1(\cE_P^\vee(-2))$ is one dimensional, we must have $H^0(End(\cE_P))=\CC$, i.e.,   
$\cE_P$ is simple. Then we deduce that $H^1(End(\cE_P))\simeq P^\vee \otimes H^1(\cE_P^\vee)$ has dimension $7\times 3=21$.\qed

\begin{lemma}
For $P$ generic,  $\cE_P$ is ordinary and 
$h^0(\cE_P^\vee(1))=8$.
\end{lemma}

\begin{proof} By Proposition \ref{syz},   $H^0(\cE_P^\vee(1)) $ coincides with 
the space $Syz_1(P)$ of linear syzygies between the quadrics in $P$, that is,
the kernel of the multiplication map $P\otimes V^\vee\ra S^3V^\vee$. 
By Theorem \ref{drezet-generation}, for $P$ generic this morphism has maximal rank, which means 
here that it is surjective. So its kernel has dimension $7\times 4-20=8$.
\end{proof}

Given a rank six vector bundle on $\PP^3$, we need in principle at least $6+3=9$ 
sections to generate it everywhere; so $8$ sections should not suffice. But note that 
 the identity   $$s(\cE_P^\vee(1))=(1+3h)/(1+h)^7 $$ 
implies $s_3(\cE_P^\vee(1))=0$, which preserves the hope that $\cE_P^\vee(1)$  be generated by global sections. Indeed, we have: 

\begin{prop} 
For $P$ general, $\cE_P^\vee(1)$ is generated by global sections.
\end{prop}

\proof For $P$ general, $\cE_P$ is ordinary and we can apply Lemma \ref{criterion},
according to which if $\cE_P^\vee(1)$ is not globally generated at $[v]$, 
there must exist a cubic  tensor $\theta \in S^3V$, independent from $v^3$, 
such that the contractions $P\rfloor\theta\subset\CC v$. 

This suggests to stratify 
$S^3V$ by the rank $\rho$ of $\iota(\theta)$. If the latter condition is realized, $[\ell]$ must be contained
in the projectivized image, a copy of $\PP^{\rho-1}$, and ${N^\perp}$ must be contained in a copy of the Grassmannian
$Gr(7,11-\rho)$, of dimension $7\times (4-\rho)$. Hence a total amount of $27-6\rho$ parameters, which is smaller 
than the dimension $21$ of $Gr(7,10)$ as soon as $\rho\ge 2$. So for ${N^\perp}$ generic $\rho\ge 2$ is impossible. But suppose that $\rho=1$. Then $\theta$
must be the cube $v^3$ of a single vector, and for $\ell$ to belong to the image of $G(\theta)$ it 
must be proportional to $v$, which means that $\bar\theta=0$ and concludes the proof. \qed 

\medskip
When $\cE_P^\vee(1)$ is generated, consider the exact sequence 
$$
0 \longrightarrow \cG_P(-2) \longrightarrow H^0(\cE_P^\vee(1))\otimes \cO_{\PP^3}  \longrightarrow \cE_P^\vee(1)
 \longrightarrow 0.$$
The rank two bundle $\cG_P\simeq\cG_P^\vee$ has Chern numbers $c_1=0$ and $c_2=3$.  

\begin{lemma}\label{gstablem}
$\cG_P$ is stable.
\end{lemma}

\proof Since $\cG_P$ has rank two and zero first Chern class, according to Hoppe's criterion for stability 
\cite{MR757476}
we just need to check that
$h^0(\cG_P(-k))=0$ for every $k\ge 0$. Because of  the isomorphism 
$\cG_P\simeq\cG_P^\vee$ 
we have an  exact sequence 
$$
0 \longrightarrow \cE_P(-k-3) \longrightarrow H^0(\cE_P^\vee(1))^\vee \otimes \mathcal{O}_\PP(-k-2) \longrightarrow \cG_P(-k)
 \longrightarrow 0.$$
 By Lemma \ref{h1}, $h^1(\cE_P(-k-3))=0$ for $k\ge 0$, hence the claim. 
\qed

\medskip
A {\it mathematical instanton}  on $\PP^3$ is a rank two stable vector bundle
$\cF$ with $c_1(\cF)=0$ and $h^1(\cF(-2))=0$. Its {\it charge} is defined as $c_2(\cF)$. 

\begin{lemma}
$\cG_P$ is a mathematical instanton of charge three.
\end{lemma}

\proof  Lemma \ref{gstablem} addresses  the stability condition, 
while the vanishing of  $h^1(\cG_P(-2))$   follows from the surjectivity
of the evaluation morphism for $ \cE_P^\vee(1)$. \qed 

\medskip
Let us dualize the  defining sequence of $\cG_P$, twist it by $\cO_{\PP^3}(1)$
and concatenate it with the  defining sequence of $\cE_P$. 
Since $\cG_P$ is self-dual, we get the free resolution 
$$
0\rightarrow \mathcal{O}_{\PP^3}(-5)\rightarrow \mathcal{O}_{\PP^3}(-3)^{\oplus 7} \rightarrow \mathcal{O}_{\PP^3}(-2)^{\oplus 8}
\rightarrow \cG_P\rightarrow 0.
$$
The matrix of linear forms defined by $\cE_P$ is exactly the one that appears  
in the middle arrow. Note that this exact sequence already appears in \cite[Proposition 4.2]{schreyer}. In \cite{gruson-skiti} it was also observed that via Serre's correspondence, $\cG_P(2)$ is associated to elliptic curves of degree seven in $\PP^3$.

\medskip
Stable rank two vector bundles $\cG$  on $\PP^3$ with Chern classes $c_1=0$ and $c_2=3$ have been extensively 
studied.  The main result of  \cite{MR611278} is that their moduli space 
$\cM(3)$ is the disjoint union of 
two irreducible components $\cM_\alpha(3)$, where $\alpha\in \ZZ/2\ZZ$ is the Atiyah-Rees invariant, defined as the parity of  $h^1(\cG(-2))$. In particular, instanton bundles are parametrized by an open subset 
of $\cM_0(3)$. Both components are smooth, rational, of dimension $21$. 


\begin{theorem}
The exists an irreducible  family of $7\times 8$ matrices of linear forms in four variables, of constant 
rank $6$, whose associated vector bundles are parametrized by the open subset 
of bundles $\cG$  inside the moduli
space of mathematical instantons of charge $3$ and Atiyah-Rees invariant $0$ on $\PP^3$, such that 
\begin{enumerate}
    \item $\cG$ has natural cohomology,
    \item  $\cG (2)$ is globally generated.
\end{enumerate}
\end{theorem}

 Recall that a vector bundle $\cG $ on $\PP ^n$ is said to have {\ em natural cohomology} if for any integer $k$ at most one cohomology group $H^{i}(\PP^n,
 \cG (k))$ is non-zero for $0\le i \le n$. 
 
\proof It follows from Lemme 1.1, Proposition 1.1.1 and sections 1.3 and 1.4 in
\cite{gruson-skiti} that an instanton bundle $\cG$ gives a matrix of constant rank 
if and only if it has natural cohomology and $\cG(2)$ is generated by global sections.\qed 

\medskip 

We   summarize the relations between the instanton bundle $\cG_P$, the rank six bundle $\cE_P$ and the matrix $M$ of linear forms, by the following diagram, where the diagonal short exact sequences are exact:

$$\xymatrix{ \cO_{\PP^3}(-2) \ar@{->}[rd]^P & & & &\cG_P(3) \\
 &  \cO_{\PP^3}^{\oplus 7}  \ar@{->}[rd]\ar@{->}[rr]^{M} & & 
 \cO_{\PP^3}(1)^{\oplus 8} \ar@{->}[ru] & \\
  & & \cE_P  \ar@{->}[ru] & & 
}$$


\medskip 
Since $P$ has codimension three, 
the initial data for the construction is essentially a tensor in $
\CC^3\otimes S^2\CC^4$, from which it is easy construct the $7\times 8$ matrix explicitly. 
In fact if $N=P^\perp\subset S^2V$  is a net of quadrics,  then 
$$H^0(\cE)=S^2V/N, \qquad H^0(\cE^\vee(1))=Ker (N^\perp \otimes V^\vee\rightarrow S^3V^\vee), $$
and the natural duality between $N^\perp$ and $S^2V/N$ yields the desired matrix.

\medskip\noindent {\it Remark.}
The study of nets of quadrics in $\PP^3$ is  a  classical topic, see for example
\cite{wall-theta}. The set of singular quadrics in such a net is parametrized 
by a plane quartic with a symmetric determinantal representation. It was already  known to Hesse that a general plane quartic admits $36$ non-equivalent such determinantal representations, corresponding to its even theta characteristics. 
This means that a general plane quartic is the characteristic curve of $36$ 
distinct nets of quadrics.


\begin{example}
Consider the net of quadrics 
$$N=\langle e_0^2-e_1^2+e_2^2, 
e_1^2+e_2^2+e_3^2, e_0e_2+e_0e_3+e_1e_3\rangle \subset S^2V.$$ The characteristic 
curve is the plane quartic of  equation 
$$\lambda^4-2\lambda^2\mu^2+\mu\nu (\mu^2-\nu^2)=0,$$
a smooth curve.
The dual space of quadrics is 
$$P=\langle x_0^2+x_1^2-x_3^2, x_0^2-x_2^2+x_3^2, x_0x_1, x_1x_2, x_2x_3,
x_0(x_2-x_3), (x_0-x_1)x_3\rangle .$$
Denote this basis of $P$ by $Q_0, Q_1, A, B, C, D, E$. The eight syzygies are
$$x_2A=x_0B, \; x_3B=x_1C, \; x_1D=(x_2-x_3)A, \; x_2E=(x_0-x_1)C,$$
$$x_0(Q_0-Q_1)=x_1A-x_0C+(x_2+2x_3)D, $$
$$x_1Q_1=x_0(A+C)-x_2B-x_3(D+E),$$
$$x_2Q_0=x_3A+x_1B-x_3C+x_0(D+E), $$
$$x_3(Q_0+Q_1)=(2x_3+x_2)A-x_2C-x_1D+(2x_0-x_1)E.$$
We get the following $7\times 8$ matrix of constant rank $6$:
$$\begin{pmatrix}
0&0&x_2&-x_0&0&0&0 \\
0&0&0&x_3&-x_1&0&0 \\
0&0&x_3-x_2&0&0&x_1&0 \\
0&0&0&0&x_1-x_0&0&x_2 \\
x_0&-x_0&-x_1&0&x_0&-x_2-2x_3&0 \\
0&x_1&-x_0&x_2&-x_0&x_3&x_3 \\
x_2&0&-x_3&-x_1&x_3&-x_0&-x_0 \\
x_3&x_3&-x_2-2x_3&0&x_2&x_1&2x_0+x_1 
\end{pmatrix}$$
\end{example}

\bigskip Our final classification is the following:

\begin{theorem}
Given an indecomposable matrix of linear forms of constant rank $6$, 
the associated 
bundle $\cE$ must be, up to switching with $\cE^\vee(1)$, of one the following types:
\begin{enumerate}
    \item ($c_1=1$) The quotient bundle $Q$ on $\PP^6$.
    \item ($c_1=2$) A Drézet bundle on $\PP^3$ defined by a net of quadrics.
    \item ($c_1=3$) The second exterior power $\wedge^2Q$ 
    on $\PP^4$. 
    \item ($c_1=3$) A Drézet bundle on $\PP^2$ defined by the cubics 
    apolar to a plane quartic. 
    \item ($c_1=3$)  A non-trivial extension of the quotient bundle $Q$ on $\PP^2$ 
    by a rank 4 non-generic Drézet bundle with $c_1=2$.
\end{enumerate}
\end{theorem}


\section{Matrices of constant rank via evaluation morphisms}

Our examples suggest a general procedure to construct spaces of 
matrices of bounded rank from a vector bundle $\cF$ on $\PP^n=\BP V$  generated by global sections.
Such $\cF$ will be the cokernel sheaf associated to  the corresponding space of  matrices
of bounded rank $e$. Such a bundle is generated by any generic space of sections of dimension at least $n+f$, where $f$ is the rank of $\cF$. So   consider $\Sigma_{e+f}\subset H^0(\PP^n, \cF)$ some generating space
of sections, of dimension $e+f$. The kernel of the evaluation map is a vector bundle 
$\cK_e$ of rank $e$, 
fitting into the exact sequence 
$$0\rightarrow \cK_e\rightarrow \Sigma_{e+f}\otimes\cO_{\PP^n}\rightarrow \cF\rightarrow 0.
$$
If $\cE:=\cK_e(1)$ has enough global sections it will give rise 
to a space of bounded rank $e$. If $\cE$ is generated by  global sections,   it will give rise to 
  a space of constant rank $e$. The composition 
$$H^0(\PP^n,\cK_e(1))\otimes \cO_{\PP^n}\rightarrow \cK_e(1)\rightarrow \Sigma_{e+f}\otimes\cO_{\PP^n}(1)$$
yields a space of matrices of linear forms of rank at most $e$, and size $a_e\times (e+f)$ for $a_e=h^0(\cK_e(1))$.
This number is potentially large:  $$ a_e\ge (n+1)(e+f)-h^0(\cF(1)),$$
and equality holds when $h^1(\cF(1))=0$,  which   is 
expected to be a condition for $\cK_e(1)$ to have nontrivial sections. We will   impose
$a_e>e$ so that the rank being bounded by $e$ is a non-trivial condition. This leads  to the conditions 
$$
\frac 1n(h^0(\cF(1))-(n+1)f)<r\leq h^0(\cF)-f.
$$
 
 \medskip
For example, suppose $\cF=Q^\vee(1)$ is the twisted dual of the quotient bundle, so $f=n$. 
If we take all the sections of $\cF$, that is for $e=\binom{n+1}{2}$, then 
$\cK_e=\wedge^2Q^\vee $ and $\cK_e(1)=\wedge^{n-2}Q$   
 is generated by global sections. In general,  $ \Sigma_{e+n}$ is a subspace 
of $\wedge^2V^\vee$, and $\cK_e$ is 
a subbundle of $\wedge^2Q^\vee$, with trivial quotient. The expected conditions for $\cK_e$ 
to admit global sections become
$$\frac{(n-1)(n+1)}{3}<e\le \binom n2.
$$
One expects that $\cK_e(1)$ remains generated by global sections when $e$ is close to $\binom{n}{2}$. How close? We have a diagram

$$\begin{CD}
    @. Q^\vee \otimes \Sigma_{e+n} @>g>> \cG \\
   @. @VVV        @VVV\\
H^0(\cK_e(1))\otimes\cO_{\PP^n} @>>> V^\vee\otimes \Sigma_{e+n} \otimes\cO_{\PP^n}
@>>> S_{2,1}V^\vee\otimes\cO_{\PP^n} \\
   @VVfV        @VVV @VVV \\
  \cK_e(1) @>>> \cO(1) \otimes \Sigma_{e+n} @>>> Q^\vee(2) \\
\end{CD}$$

\smallskip\noindent
where $\cG$, which has rank $\frac {n(n^2+3n-1)}3$,  is defined by the rightmost 
vertical exact sequence.

By the snake lemma, if the sequence of trivial bundles is exact, $f$ is 
surjective exactly when $g$ is surjective, or equivalently when $g^t$ is injective. 
In order to decide whether it is, note that there is an exact sequence 
$$0\rightarrow \cG^\vee\rightarrow \wedge^2V\otimes Q\rightarrow \wedge^3Q\rightarrow 0.
$$
To see this, over a point $[v]$ choose a splitting $V=H\op \ell$ as before with $\ell=\BC v$.
then over the point $\cG\isom S_{21}H^\vee\op S^2H^\vee\ot \ell^\vee$, and thus
$\cG^\vee$ maps to $\La 2 V\ot Q \isom \cG^\vee\op \La 3 H$, and the latter term is
the fiber of $\La 3Q$.

  The above discussion implies:

\begin{lemma} The evaluation morphism $H^0(\cK_e(1)) \otimes\cO_{\PP^n} \rightarrow \cK_e(1)$ is surjective exactly 
when the composition 
$$ \Sigma_{e+n}^\perp \otimes Q\hookrightarrow \wedge^2V\otimes Q\rightarrow \wedge^3Q$$
is everywhere injective.
\end{lemma} 

\begin{theorem} For $n\geq 6$ and  $\frac{n(n-1)}2\geq e> \frac{(n-1)^2}2$,  a general
subspace $\Sigma_{e+n}\subset \La 2 V$ of dimension $e+n$ defines a matrix of linear forms
of constant rank $e$ of size  $a_e\times b_e$, where
 $$
 a_e=e+n, \qquad b_e=e(n+1)-2\binom{n+1}3.
$$
The corresponding rank $e$ vector bundle $\cE$ has $c_1(\cE)=e-n+1$.
\end{theorem}

\begin{proof}
Set $\ell=\binom{n+1}2-(n+e)$, the codimension of $\Sigma=\Sigma_{e+n}$, and let
    $\Theta_\ell=\Sigma_{k+n}^\perp$, which is  a generic $\ell$-dimensional subspace of $\wedge^2V$. 

First consider the case $\ell=1$. 
Suppose $\Theta_1\otimes Q\rightarrow \wedge^3Q$ is not injective at some point $[v]$, and that $\omega\otimes \bar{u}$
is sent to zero for a generator $\omega$ of $\Theta_1$ and a vector $u\in V$, not colinear to $v$. This means that 
$\omega\wedge u$ belongs to the kernel of the projection $\wedge^3V\rightarrow \wedge^3Q$, that is $v\wedge(\wedge^2V)$. But then, since $u$ and $v$ are independent, this implies that $\omega=u\wedge\alpha+v\wedge\beta$ for some vectors $\alpha,
\beta$. In particular $\omega$ has rank at most four, so $\Theta_1$ cannot be generic for $n\ge 6$. 

Now consider the general case, and suppose that at some point $[v]$, an element $\omega_1\otimes\bar{u}_1+\cdots + \omega_\ell\otimes\bar{u}_\ell$ of $\Theta_\ell\otimes Q$ is sent to zero in  $\wedge^3Q$. Since $\ell<n$ we 
may suppose $\omega_1, \ldots, \omega_\ell$ to be a basis of $\Theta_\ell$, and  $\bar{u}_1, \ldots , \bar{u}_\ell$
to be independent in $Q$ -- otherwise we are reduced to considering spaces of smaller dimension. As before,
this means that $\omega_1\wedge u_1+\cdots + \omega_\ell\wedge u_\ell=v\wedge \sigma$ for some $\sigma\in 
\wedge^2V$. Completing $u_1,\ldots ,u_\ell, v$ into a basis of $V$ by vectors $w_1,\ldots, w_m$, and decomposing
into this basis, we immediately see that $\omega_1, \ldots, \omega_\ell$ cannot contain any term of the form 
$w_i\wedge w_j$. In other words, 
$$\Theta_\ell \subset A_{\ell+1}\wedge V, \qquad A_{\ell+1}:=\langle u_1,\ldots ,u_\ell, v\rangle.$$
We claim this forces $\Theta_\ell$ to be non-generic, just by counting dimensions. Indeed, $\Theta_\ell$ 
must belong to a subvariety $Z$ of the Grassmannian $Gr(\ell,\wedge^2V)$ which is dominated by the total
space of the Grassmann bundle $Gr(\ell,A_{\ell+1}\wedge V)$ over $Gr(\ell+1,V)$, where $A_{\ell+1}$ denotes the
tautological bundle on the latter Grassmannian. Our conclusion holds as soon as $\dim Z<\dim Gr(\ell,\wedge^2V)$, 
and it is a straightforward to check that this is the case in the indicated range.
\end{proof}
 






    \bibliographystyle{amsplain}

\bibliography{Lmatrix}

\end{document}